\documentclass[twoside,11pt]{article}

\usepackage{jmlr2e}
\usepackage{amsmath}
\usepackage{algorithm}

\newtheorem{Def}{Definition}
\newtheorem{Assumption}{Assumption}
\newtheorem{Thm}{Theorem}

\newtheorem{Lem}{Lemma}
\newtheorem{Cor}{Corollary}
\newtheorem{Rem}{Remark}

\newcommand*{\argmin}{argmin}
\newcommand*{\tran}{^{\mkern-1.5mu\mathsf{T}}}

\usepackage{lastpage}
\jmlrheading{20}{2019}{1-\pageref{LastPage}}{10/16; Revised 8/18}{7/19}{16-494}{Hao Yu and Michael Neely}
\ShortHeadings{Online Convex Opt with Long Term Constraints}{Yu and Neely}

\begin{document}	

\title{A Low Complexity Algorithm with $O(\sqrt{T})$ Regret and $O(1)$ Constraint Violations for Online Convex Optimization with Long Term Constraints}

\author{\name Hao Yu \email eeyuhao@gmail.com \\ 
	\addr  Department of Electrical Engineering \\
	University of Southern California \\ 
	Los Angeles, CA,  90089-2565, USA 
	\AND 
	Michael J.\ Neely \email mjneely@usc.edu \\
	\addr  Department of Electrical Engineering \\
	University of Southern California \\ 
	Los Angeles, CA,  90089-2565, USA }

\editor{Csaba Szepesvari}

\maketitle

\begin{abstract}
This paper considers online convex optimization over a complicated constraint set, which typically consists of multiple functional constraints and a set constraint.  The conventional online projection algorithm \citep{Zinkevich03ICML} can be difficult to implement due to the potentially high computation complexity of the projection operation. In this paper, we relax the functional constraints by allowing them to be violated at each round but still requiring them to be satisfied in the long term. This type of relaxed online convex optimization (with long term constraints) was first considered in \citet{Mehdavi12JMLR}. That prior work proposes an algorithm to achieve $O(\sqrt{T})$ regret and $O(T^{3/4})$ constraint violations for general problems and another algorithm to achieve an $O(T^{2/3})$ bound for both regret and constraint violations when the constraint set can be described by a finite number of linear constraints.  A recent extension in \citet{Jenatton16ICML} can achieve $O(T^{\max\{\theta,1-\theta\}})$ regret and $O(T^{1-\theta/2})$ constraint violations where $\theta\in (0,1)$. The current paper proposes a new simple algorithm that yields improved performance in comparison to prior works. The new algorithm achieves an $O(\sqrt{T})$ regret bound with $O(1)$ constraint violations. 
\end{abstract}


\section{Introduction} 
Online optimization and learning is a multi-round process of making decisions in the presence of uncertainty, where a decision strategy should generally adapt decisions based on results of previous rounds \citep{book_PredictionLearningGames}. Online convex optimization is an important subclass of these problems where the received loss function is convex with respect to the decision. At each round of online convex optimization, the decision maker is required to choose $\mathbf{x}(t)$ from a known convex set $\mathcal{X}$. After that, the convex loss function $f^{t}(\mathbf{x}(t))$ is disclosed to the decision maker. Note that the loss function can change arbitrarily every round $t$, with no probabilistic model imposed on the changes. 

The goal of an online convex optimization algorithm is to select a good sequence $\mathbf{x}(t)$ such that the accumulated loss $\sum_{t=1}^{T} f^{t} (\mathbf{x}(t))$ is competitive with the loss of any fixed $\mathbf{x}\in \mathcal{X}$.  To capture this, the $T$-round regret with respect to the best fixed decision is defined as follows:
\begin{align}
\text{Regret}_T  = \sum_{t=1}^T f^{t}(\mathbf{x}(t)) -  \min_{\mathbf{x}\in \mathcal{X}}\sum_{t=1}^T f^{t} (\mathbf{x}).
\end{align}

The best fixed decision in hindsight $\mathbf{x}^{\ast} = \argmin_{\mathbf{x}\in \mathcal{X}}\sum_{t=1}^T f^{t} (\mathbf{x})$ typically cannot be implemented.  That is because it would need to be determined before the start of the first round, and this would require knowledge of the future $f^{t}(\cdot)$ functions for all $t\in\{1,2,\ldots, T\}$. However, to avoid  being embarrassed by the situation where our performance is significantly exceeded by a stubborn decision maker guessing $\mathbf{x}^{\ast}$ correctly by luck, a desired learning algorithm should have a small regret. Specifically, we desire a learning algorithm for which $\text{Regret}_{T}$ grows sub-linearly with respect to $T$, i.e., the difference of average loss tends to zero as $T$ goes to infinity when comparing the dynamic learning algorithm and a lucky stubborn decision maker. 

For online convex optimization with loss functions that are convex and have bounded gradients\footnote{In fact, Zinkevich's algorithm in \citep{Zinkevich03ICML} can be extended to treat non-differentiable convex loss functions by replacing the gradient with the subgradient. The same $O(\sqrt{T})$ regret can be obtained as long as the convex loss functions have bounded subgradients. This paper also has the bounded gradient assumption in Assumption \ref{as:basic}. This is solely for simplicity of the presentation. In fact, none of the results in this paper require the differentiability of loss functions. If any loss function is non-differentiable, we could replace the gradient with the subgradient and obtain the same regret and constraint violation bounds by replacing the bounded gradient assumption with the bounded subgradient assumption.}, the best known regret is $O(\sqrt{T})$ and is attained by a simple online gradient descent algorithm \citep{Zinkevich03ICML}. At the end of each round $t$, Zinkevich's algorithm updates the decision for the next round $t+1$ by 
\begin{align}
\mathbf{x}(t+1) = \mathcal{P}_{\mathcal{X}} \left[\mathbf{x}(t) - \gamma \nabla f^{t}(\mathbf{x}(t))\right], \label{eq:OGD}
\end{align}
where $\mathcal{P}_{\mathcal{X}}[\cdot]$ represents the projection onto convex set $\mathcal{X}$ and $\gamma$ is the step size.

\citet{Hazan07ML}  shows that better regret is possible under a more restrictive strong convexity assumption. However, \citet{Hazan07ML} also shows that $\Omega(\sqrt{T})$ regret is unavoidable with no additional assumption. 

In the case when $\mathcal{X}$ is a simple set, e.g., a box constraint, the projection $ \mathcal{P}_{\mathcal{X}}[\cdot]$ is simple to compute and often has a closed form solution. However, if set $\mathcal{X}$ is complicated, e.g., set $\mathcal{X}$ is described via a number of functional constraints as $\mathcal{X} = \{\mathbf{x}\in \mathcal{X}_{0}: g_{k}(\mathbf{x})\leq 0, i\in\{1,2,\ldots, m\}\}$, then  equation \eqref{eq:OGD} requires to solve the following convex program:
\begin{align}
\text{minimize:} \quad & \Vert \mathbf{x} - [\mathbf{x}(t) - \gamma \nabla f^{t}(\mathbf{x}(t))]\Vert^2 \label{eq:proj-problem-1}\\
\text{such that:} \quad  &  g_{k}(\mathbf{x}) \leq  0, \forall k\in\{1,2,\ldots,m\} \label{eq:proj-problem-2}\\
			 &  \mathbf{x}\in \mathcal{X}_{0} \in \mathbb{R}^{n}\label{eq:proj-problem-3}
\end{align}
which can yield heavy computation and/or storage burden at each round. For instance, the interior point method (or other Newton-type methods) is an iterative algorithm and takes a number of iterations to approach the solution to the above convex program. The computation and memory space complexity at each iteration is between $O(n^{2})$ and $O(n^{3})$, where $n$ is the dimension of $\mathbf{x}$.

As an attempt to reduce the projection complexity, \citet{Hazan12ICML} use the Frank-Wolfe technique to replace the quadratic convex program \eqref{eq:proj-problem-1}-\eqref{eq:proj-problem-3} with simpler linear optimization with the same constraints \eqref{eq:proj-problem-2}-\eqref{eq:proj-problem-3}. To completely circumvent the computational challenge due to constraint \eqref{eq:proj-problem-2} in the projection operator, a variation of the standard online convex optimization, also known as online convex optimization with long term constraints, is first considered by \citet{Mehdavi12JMLR}.  In this variation, complicated functional constraints $g_{k}(\mathbf{x})\leq 0$ are relaxed to be soft long term constraints.   That is, we do not require $\mathbf{x}(t)\in \mathcal{X}_{0}$ to satisfy $g_{k}(\mathbf{x}(t))\leq 0$ at each round, but only require that $\sum_{t=1}^{T} g_{k} (\mathbf{x}(t))$, called constraint violations, grows sub-linearly. \citet{Mehdavi12JMLR} proposes two algorithms such that one achieves $O(\sqrt{T})$ regret and $O(T^{3/4})$ constraint violations; and the other achieves $O(T^{2/3})$ for both regret and constraint violations when the set $\mathcal{X}$ can be represented by linear constraints. Further, \citet{Mehdavi12JMLR} posed an open question of whether there exists a low complexity algorithm with an $O(\sqrt{T})$ bound on the regret and a better bound than $O(T^{3/4})$ on the constraint violations. \citet{Jenatton16ICML} recently extends the algorithm of \citet{Mehdavi12JMLR} to achieve $O(T^{\max\{\theta,1-\theta\}})$   regret and $O(T^{1-\theta/2})$ constraint violations where $\theta\in (0,1)$ is a user-defined tradeoff parameter. By choosing $\theta=1/2$ or $\theta=2/3$, the $[O(\sqrt{T}), O(T^{3/4})]$ or $[O(T^{2/3}),O(T^{2/3})]$ regret and constraint violations of \citet{Mehdavi12JMLR} are recovered. It is easy to observe that the best regret or constraint violations in \citet{Jenatton16ICML} are $O(\sqrt{T})$ under different $\beta$ values. However, the algorithm of \citet{Jenatton16ICML}  can not achieve $O(\sqrt{T})$ regret and $O(\sqrt{T})$ constraint violations simultaneously.  

The current paper proposes a new algorithm that can achieve $O(\sqrt{T})$ regret and $O(1)$ constraint violations that do not grow with $T$; and hence yields improved performance in comparison to prior works \citep{Mehdavi12JMLR,Jenatton16ICML}. The algorithm is the first to reduce the complexity associated with multiple constraints while maintaining $O(\sqrt{T})$ regret and achieving a constraint violation bound strictly better than $O(T^{3/4})$. Hence, we give a positive answer to the open question posed by \citet{Mehdavi12JMLR}. The new algorithm is related to a recent technique we developed for deterministic convex programming with a fixed objective function \citep{YuNeely17SIOPT} and the drift-plus-penalty technique for stochastic optimization in dynamic queue networks \citep{book_Neely10}. Our other paper \citep{YuNeelyWei17NIPS} developed another algorithm with $O(\sqrt{T})$ regret and weaker $O(\sqrt{T})$ constraint violations for online convex optimization with stochastic constraints that include long term constraints as special cases. 

Many engineering problems can be directly formulated as online convex optimization with long term constraints. 
For example, problems with energy or monetary constraints often define these in terms of long term time averages rather than instantaneous constraints. In general, we assume that instantaneous constraints are incorporated into the set $\mathcal{X}_{0}$; and long term constraints are represented via functional constraints $g_{k}(\mathbf{x})$. Two example problems are given as follows. More examples can be found in \citet{Mehdavi12JMLR} and \citet{Jenatton16ICML}.
\begin{itemize}
\item In the application of online display advertising \citep{Goldfarb11DisplayAd,Ghosh09WINE}, the publisher needs to iteratively allocate ``impressions'' to advertisers to optimize some online concave utilities for each advertiser. The utility is typically unknown when the decision is made but can be inferred later by observing user click behaviors under the given allocations. Since each advertiser usually specifies a certain budget for a period, the ``impressions'' should be allocated to maximize advertisers' long term utilities subject to long term budget constraints. 
\item In the application of network routing in a neutral or adversarial  environment, the decision maker needs to iteratively make routing decisions to maximize network utilities. Furthermore,  link quality can vary after each routing decision is made.  The routing decisions should satisfy the long term flow conservation constraint at each intermediate node so  that queues do not overflow. 
\end{itemize}

\section{Online Convex Optimization with Long Term Constraints} \label{sec:formulation}
This section introduces the problem of online convex optimization with long term constraints and presents our new algorithm.
\subsection{Online Convex Optimization with Long Term Constraints}
Let $\mathcal{X}_{0}$ be a closed convex set and $g_{k}(\mathbf{x}), k\in\{1,2,\ldots, m\}$ be continuous convex functions. Denote the stacked vector of multiple functions $g_{1}(\mathbf{x}), \ldots, g_{m}(\mathbf{x})$ as $\mathbf{g}(\mathbf{x})= [g_{1}(\mathbf{x}), \ldots, g_{m}(\mathbf{x})]\tran$. Define $\mathcal{X}=\{\mathbf{x}\in \mathcal{X}_{0}: g_{k}(\mathbf{x})\leq 0, k\in \{1,2,\ldots, m\}\}$.  Let $f^{t}(\mathbf{x})$ be a sequence of continuous convex loss functions which are determined by nature (or by an adversary) such that $f^{t}(\mathbf{x})$ is unknown to the decision maker until the end of round $t$. For any sequence $\mathbf{x}(t)$ yielded by an online algorithm, define $\sum_{t=1}^{T} f^{t}(\mathbf{x}(t)) - \min_{\mathbf{x}\in\mathcal{X}} \sum_{t=1}^{T}f^{t}(\mathbf{x})$ as the regret and $\sum_{t=1}^{T} g_{k}(\mathbf{x}(t)), k\in\{1,2,\ldots, m\}$ as the constraint violations.  The goal of online convex optimization with long term constraints is to choose $\mathbf{x}(t)\in \mathcal{X}_{0}$ for each round $t$ such that both the regret and the constraint violations grow sub-linearly with respect to $T$.  Throughout this paper, we use $\Vert \cdot \Vert$ to denote the Euclidean norm.
\begin{Assumption}\label{as:basic}~
\begin{itemize}
\item The loss functions have bounded gradients on $\mathcal{X}_{0}$. That is, there exists $D>0$ such that $\Vert \nabla f^{t} (\mathbf{x})\Vert \leq D$ for all $\mathbf{x}\in \mathcal{X}_{0}$ and all $t$.
\item  There exists a constant $\beta$ such that $\Vert \mathbf{g}(\mathbf{x}) - \mathbf{g}(\mathbf{y})\Vert \leq \beta \Vert \mathbf{x} - \mathbf{y}\Vert$ for all $\mathbf{x}, \mathbf{y}\in \mathcal{X}_{0}$, i.e., $\mathbf{g}(\mathbf{x})$ is Lipschitz continuous with modulus $\beta$.
\end{itemize}
\end{Assumption}

\begin{Assumption} \label{as:bounded-g} There exists a constant $G$ such that $\Vert \mathbf{g}(\mathbf{x})\Vert \leq G$ for all $\mathbf{x}\in \mathcal{X}_0$.  
\end{Assumption}

\begin{Assumption} \label{as:bounded-set} There exists a constant $R$ such that $\Vert \mathbf{x} - \mathbf{y}\Vert\leq R$ for all $\mathbf{x}, \mathbf{y}\in \mathcal{X}_{0}$.
\end{Assumption}

Note that if $\mathcal{X}_{0}$ is bounded, then the existence of $G$ follows directly from the compactness of set $\mathcal{X}_{0}$ and the continuity of $\mathbf{g}(\mathbf{x})$ and the existence of $R$ follows directly from the boundedness of set $\mathcal{X}_{0}$.
  
\begin{Assumption}\label{as:interior-point}
There exists $\epsilon>0$ and $\hat{\mathbf{x}}\in \mathcal{X}_{0}$ such that $g_{k}(\hat{\mathbf{x}}) \leq -\epsilon$ for all $k\in\{1,2,\ldots, m\}$.
\end{Assumption}

Assumption \ref{as:interior-point}, known as the Slater condition or the interior point condition, is a mild assumption in convex optimization \citep{book_ConvexOptimization}.
   
In Sections \ref{sec:formulation}-\ref{sec:analysis}, we shall propose a new algorithm to achieve $O(\sqrt{T})$ regret and $O(1)$ constraint violations for online convex optimization with long term constraints under Assumptions \ref{as:basic}-\ref{as:interior-point} by assuming time horizon $T$ is known in advance.  In Section \ref{sec:extensions}, further extensions on dealing with unknown time horizon $T$ and relaxing Assumptions \ref{as:bounded-set}-\ref{as:interior-point} will be discussed.

\subsection{New Algorithm}

Define $\tilde{\mathbf{g}}(\mathbf{x}) = \gamma \mathbf{g}(\mathbf{x})$ where $\gamma>0$ is an algorithm parameter. Note that each $\tilde{g}_k(\mathbf{x})$ is still a convex function and $\tilde{\mathbf{g}}(\mathbf{x}) \leq \mathbf{0}$ if and only if $\mathbf{g}(\mathbf{x})\leq \mathbf{0}$. The next lemma follows directly.
\begin{Lem}\label{lm:scaled-constant}
If online convex optimization with long term constraints satisfies Assumptions \ref{as:basic}-\ref{as:interior-point}, then
\begin{itemize}
\item $\Vert \tilde{\mathbf{g}}(\mathbf{x}) - \tilde{\mathbf{g}}(\mathbf{y})\Vert \leq \gamma \beta \Vert \mathbf{x} - \mathbf{y}\Vert$ for all $\mathbf{x}, \mathbf{y}\in \mathcal{X}_{0}$.
\item $\Vert \tilde{\mathbf{g}}(\mathbf{x})\Vert \leq \gamma G$ for all $\mathbf{x}\in \mathcal{X}_0$.
\item $\tilde{g}_{k}(\hat{\mathbf{x}}) \leq -\gamma \epsilon$ for all $k\in\{1,2,\ldots, m\}$.
\end{itemize}
\end{Lem}

Now consider the following algorithm described in Algorithm \ref{alg:new-alg}. This algorithm chooses $\mathbf{x}(t+1)$ as the decision for round $t+1$ based on $f^{t} (\cdot)$ without knowing the cost function $f^{t+1}(\cdot)$.  In this paper, we show that if the parameters $\gamma$ and $\alpha$ are chosen to satisfy $\gamma = T^{1/4}$ and $\alpha = \frac{1}{2}[\beta^{2} + 1]\sqrt{T}$, then Algorithm \ref{alg:new-alg} achieves an $O(\sqrt{T})$ regret bound with $O(1)$ constraint violations.
\begin{algorithm} 
\caption{}
\label{alg:new-alg}
Let $\gamma, \alpha>0$ be constant parameters. Initialize $Q_{k}(0) = 0, \forall k\in\{1,2,\ldots,m\}$.  Choose arbitrary $\mathbf{x}(1) \in \mathcal{X}_{0}$.  At the end of each round $t\in\{1,2,3,\ldots\}$, observe $f^{t}(\cdot)$ and do the following:
\begin{itemize}
\item Update virtual queue vector $\mathbf{Q}(t)$ via 
\begin{align*}
Q_{k}(t) = \max\left\{-\tilde{g}_{k}(\mathbf{x}(t)), Q_{k}(t-1) + \tilde{g}_{k}(\mathbf{x}(t))\right\}, \forall k\in\{1,2,\ldots,m\}.
\end{align*}
\item  Choose $\mathbf{x}(t+1)$ that solves
 \begin{align*}
 \min_{\mathbf{x}\in \mathcal{X}_{0}}\left\{ [\nabla f^{t}(\mathbf{x}(t))]\tran [\mathbf{x} - \mathbf{x}(t)] + [\mathbf{Q}(t) + \tilde{\mathbf{g}}(\mathbf{x}(t))]\tran\tilde{\mathbf{g}}(\mathbf{x}) + \alpha \Vert \mathbf{x} - \mathbf{x}(t)\Vert^{2} \right\}
\end{align*}
as the decision for the next round $t+1$, where $\nabla f^{t}(\mathbf{x}(t))$ is the gradient of $f^{t}(\mathbf{x})$ at point $\mathbf{x}=\mathbf{x}(t)$.
\end{itemize}
\end{algorithm}

This algorithm introduces a virtual queue vector for constraint functions. The update equation of this virtual queue vector is similar to an algorithm recently developed by us for deterministic convex programs (with a fixed and known objective function) in \citet{YuNeely17SIOPT}. However, the update for $\mathbf{x}(t+1)$ is different from \citet{YuNeely17SIOPT}.  The role of $\mathbf{Q}(t)$ is similar to a Lagrange multiplier vector and its value is like a  ``queue backlog" of constraint violations.  By introducing the virtual queue vector, we can characterize the regret and constraint violations for the new algorithm through the analysis of a drift-plus-penalty expression. The analysis of a drift-plus-penalty expression was originally considered in stochastic optimization for dynamic queueing systems where the decision at each round is made by observing the instantaneous cost function  that changes in an i.i.d. manner \citep{PhD_Thesis_Neely, book_Neely10}. The algorithm developed in this paper is different from the conventional drift-plus-penalty algorithm in both the decision update and the virtual queue update.  However, it turns out that the analysis of the drift-plus-penalty expression is also the key step to analyze the regret and the constraint violation bounds for online convex optimization with long term constraints.

Because of the $\Vert \mathbf{x} - \mathbf{x}(t)\Vert^{2}$ term, the choice of $\mathbf{x}(t+1)$ in Algorithm \ref{alg:new-alg} involves  minimization of a \emph{strongly convex function} (strong convexity is formally defined in the next section).  
If the constraint functions $\mathbf{g}(\mathbf{x})$ are separable (or equivalently, $\tilde{\mathbf{g}}(\mathbf{x})$ are separable) with respect to components or blocks of $\mathbf{x}$, e.g., $\mathbf{g}(\mathbf{x}) = \sum_{i=1}^{n}\mathbf{g}^{(i)}(x_{i})$ or $\mathbf{g}(\mathbf{x}) = \mathbf{A}\mathbf{x} -\mathbf{b}$, then the primal updates for $\mathbf{x}(t+1)$ can be decomposed into several smaller independent subproblems, each of which only involves a component or block of $\mathbf{x}(t+1)$.  The next lemma further shows that the update of $\mathbf{x}(t+1)$ follows a simple gradient update in the case when $\mathbf{g}(\mathbf{x})$ is linear. 

\begin{Lem}
If $\mathbf{g}(\mathbf{x})$ is affine, i.e, $\mathbf{g}(\mathbf{x}) = \mathbf{A}\mathbf{x}+\mathbf{b}$ for some matrix $\mathbf{A}$ and vector $\mathbf{b}$, then the update of $\mathbf{x}(t+1)$ at each round in Algorithm \ref{alg:new-alg} is given by
\begin{align*}
\mathbf{x}(t+1) = \mathcal{P}_{\mathcal{X}_{0}}\Big[\mathbf{x}(t) - \frac{1}{2\alpha} \mathbf{d}(t)\Big],
\end{align*}
where $\mathbf{d}(t) = \nabla f^{t}(\mathbf{x}(t)) + \sum_{k=1}^{m} [Q_{k}(t) + \tilde{g}_{k}(\mathbf{x}(t))] \nabla \tilde{g}_{k}(\mathbf{x}(t))$ and $\mathcal{P}_{\mathcal{X}_0}[\cdot]$ denotes projection onto set $\mathcal{X}_0$.
\end{Lem}

\begin{proof}
Fix $t\geq \{0,1,\ldots\}$.  Note that $\mathbf{d}(t)$ is a constant vector in the update of $\mathbf{x}(t+1)$. The projection operator can be interpreted as an optimization problem as follows:

\begin{align}
&\mathbf{x}(t+1) =  \mathcal{P}_{\mathcal{X}_{0}} \Big[ \mathbf{x}(t) - \frac{1}{2\alpha} \mathbf{d}(t)\Big] \nonumber\\
\overset{(a)}{\Leftrightarrow}\quad& \mathbf{x}(t+1)= \argmin_{\mathbf{x}\in \mathcal{X}_{0}} \Big[ \big\Vert \mathbf{x} - [\mathbf{x}(t)- \frac{1}{2\alpha} \mathbf{d}(t)] \big\Vert^{2}\Big] \nonumber\\
\Leftrightarrow\quad& \mathbf{x}(t+1) = \argmin_{\mathbf{x}\in \mathcal{X}_{0}} \Big[ \Vert \mathbf{x} -\mathbf{x}(t) \Vert^{2} + \frac{1}{\alpha} \mathbf{d}\tran(t)[\mathbf{x} -\mathbf{x}(t)] + \frac{1}{4\alpha^{2}} \Vert \mathbf{d}(t) \Vert^{2}\Big] \nonumber\\
\overset{(b)}{\Leftrightarrow} \quad&  \mathbf{x}(t+1) = \argmin_{\mathbf{x}\in \mathcal{X}_{0}} \Big[   \sum_{k=1}^m [ Q_{k}(t+1) + \tilde{g}_{k} (\mathbf{x}(t))] \tilde{g}_{k}(\mathbf{x}(t))+   \mathbf{d}\tran(t)[\mathbf{x} -\mathbf{x}(t)] \nonumber \\ &\qquad\qquad\qquad\qquad\qquad~ +   \alpha\Vert \mathbf{x} -\mathbf{x}(t) \Vert^{2} \Big] \nonumber
\end{align}
\begin{align}
\overset{(c)}{\Leftrightarrow} \quad& \mathbf{x}(t+1) = \argmin_{\mathbf{x}\in \mathcal{X}_{0}} \Big[ [\nabla f^{t}(\mathbf{x}(t))]\tran [\mathbf{x} -\mathbf{x}(t)] + \sum_{k=1}^m [ Q_{k}(t) + \tilde{g}_{k} (\mathbf{x}(t))] \tilde{g}_{k}(\mathbf{x}(t)) \nonumber\\ &\qquad\qquad\qquad\qquad~  +   \sum_{k=1}^{m} [Q_{k}(t)+\tilde{g}_{k}(\mathbf{x}(t))] [\nabla \tilde{g}_{k}(\mathbf{x}(t))]\tran[\mathbf{x} -\mathbf{x}(t)] +  \alpha \Vert \mathbf{x} -\mathbf{x}(t) \Vert^{2} \Big] \nonumber\\
 \overset{(d)}{\Leftrightarrow} \quad& \mathbf{x}(t+1) = \argmin_{\mathbf{x}\in \mathcal{X}_{0}} \left[ [\nabla f^{t}(\mathbf{x}(t))]\tran [\mathbf{x} -\mathbf{x}(t)]  + [\mathbf{Q}(t) + \tilde{\mathbf{g}}(\mathbf{x}(t))]\tran \tilde{\mathbf{g}}(\mathbf{x}) +  \alpha\Vert \mathbf{x} -\mathbf{x}(t) \Vert^{2} \right], \nonumber
\end{align}
where (a) follows from the definition of the projection onto a convex set; (b) follows from the fact the minimizing solution does not change when we remove constant term $\frac{1}{4\alpha^{2}} \Vert \mathbf{d}(t)\Vert^{2}$, multiply positive constant $\alpha$ and add constant term $\sum_{k=1}^m [ Q_{k}(t) + \tilde{g}_{k} (\mathbf{x}(t))] \tilde{g}_{k}(\mathbf{x}(t)) $ in the objective function; (c) follows from the definition of $\mathbf{d}(t)$; and (d) follows from the identity $[\mathbf{Q}(t) + \tilde{\mathbf{g}}(\mathbf{x}(t))]\tran \tilde{\mathbf{g}}(\mathbf{x}) = \sum_{k=1}^m [ Q_{k}(t) + \tilde{g}_{k} (\mathbf{x}(t))] \tilde{g}_{k}(\mathbf{x}(t))+ \sum_{k=1}^{m} [Q_{k}(t)+\tilde{g}_{k}(\mathbf{x}(t))] [\nabla \tilde{g}_{k}(\mathbf{x}(t))]\tran[\mathbf{x} -\mathbf{x}(t)] $ for any $\mathbf{x}\in \mathbb{R}^{n}$, which further follows from the linearity of $\tilde{\mathbf{g}}(\mathbf{x})$.
\end{proof}

\subsection{Properties of Virtual Queues in Algorithm \ref{alg:new-alg}} 

In this  subsection, we summarize important properties for virtual queues introduced in Algorithm \ref{alg:new-alg}. The virtual queue properties in this subsection are similar to those we developed in another context in \citet{YuNeely17SIOPT}. However, the virtual 
queue of the current paper is slightly different and we give proofs for 
completeness.  These properties come from the algorithm design and rely on none of Assumptions \ref{as:basic}-\ref{as:interior-point}.

\begin{Lem}\label{lm:virtual-queue} In Algorithm \ref{alg:new-alg}, we have
\begin{enumerate}
\item At each round $t\in\{0,1,2,\ldots\}$, $Q_k(t)\geq 0$ for all $k\in\{1,2,\ldots,m\}$.
\item At each round $t\in\{1,2,\ldots\}$, $Q_{k}(t) + \tilde{g}_{k}(\mathbf{x}(t))\geq 0$ for all $k\in\{1,2\ldots, m\}$.
\item At round $t=1$, $\Vert \mathbf{Q}(1)\Vert^2 = \Vert \tilde{\mathbf{g}}(\mathbf{x}(1))\Vert^2$. At each round $t\in\{2,3,\ldots\}$,  $\Vert \mathbf{Q}(t)\Vert^2 \geq \Vert \tilde{\mathbf{g}}(\mathbf{x}(t))\Vert^2$.
\item At each round $t\in\{1,2,\ldots\}$,  $\Vert \mathbf{Q}(t)\Vert \leq \Vert \mathbf{Q}(t-1)\Vert  + \Vert \tilde{\mathbf{g}}(\mathbf{x}(t))\Vert$.
\end{enumerate}
\end{Lem}
\begin{proof}~
\begin{enumerate}
\item Fix $k\in\{1,2,\ldots, m\}$.  The proof is by induction.  Note that $Q_k(0) = 0$ by initialization. Assume $Q_k(t)\geq 0$ for some $t \in \{0, 1, 2, \ldots\}$.  We now prove $Q_k(t+1)\geq 0$. If $\tilde{g}_{k}(\mathbf{x}(t+1)) \geq 0$, the virtual queue update equation of Algorithm \ref{alg:new-alg} gives:
\begin{align*}
Q_k(t+1) &= \max\{-\tilde{g}_{k}(\mathbf{x}(t+1)), Q_{k}(t) + \tilde{g}_{k}(\mathbf{x}(t+1))\} \geq Q_{k}(t) + \tilde{g}_{k}(\mathbf{x}(t+1)) \geq 0.
\end{align*} 
 On the other hand, if $\tilde{g}_{k}(\mathbf{x}(t+1))<0$, then $Q_k(t+1) = \max\{-\tilde{g}_{k}(\mathbf{x}(t+1)), Q_{k}(t) + \tilde{g}_{k}(\mathbf{x}(t+1))\}\geq -\tilde{g}_{k}(\mathbf{x}(t+1)) > 0$.  Thus, in both cases we have $Q_k(t+1) \geq 0$. 
\item Fix $k\in\{1,2,\ldots, m\}$. Fix $t\in\{0,1,\ldots\}$. By the virtual queue update equation, we have $Q_{k}(t) = \max\{-\tilde{g}_{k}(\mathbf{x}(t)), Q_{k}(t-1) + \tilde{g}_{k}(\mathbf{x}(t))\}\geq -\tilde{g}_{k}(\mathbf{x}(t))$, which implies that $Q_{k}(t) + \tilde{g}_{k}(\mathbf{x}(t)) \geq 0$.
\item  Since we initialize $Q_{k}(0) = \mathbf{0}, \forall k$, we have $Q_{k}(1) = \max\{-\tilde{g}_{k}(\mathbf{x}(1)),\tilde{g}_{k}(\mathbf{x}(1))\} = \vert \tilde{g}_{k}(\mathbf{x}(1))\vert, \forall k$. Thus, $\Vert \mathbf{Q}(1)\Vert^2 = \Vert \tilde{\mathbf{g}}(\mathbf{x}(1))\Vert^2$. Fix $t\in\{2,3,\ldots\}$ and  $k\in\{1,2,\ldots, m\}$. If $\tilde{g}_{k}(\mathbf{x}(t)) \geq 0$, then 
\begin{align*}
Q_k(t) &= \max\{-\tilde{g}_{k}(\mathbf{x}(t)), Q_{k}(t-1) + \tilde{g}_{k}(\mathbf{x}(t))\} \\
&\geq Q_{k}(t-1) + \tilde{g}_{k}(\mathbf{x}(t))  \\
&\overset{(a)}{\geq} \tilde{g}_{k}(\mathbf{x}(t)) = |\tilde{g}_{k}(\mathbf{x}(t))|,
\end{align*}  
where (a) follows from part 1. On the other hand, if  $\tilde{g}_{k}(\mathbf{x}(t)) < 0$, then $Q_k(t) = \max\{-\tilde{g}_{k}(\mathbf{x}(t)), Q_{k}(t-1) + \tilde{g}_{k}(\mathbf{x}(t))\} \geq -\tilde{g}_{k}(\mathbf{x}(t)) = |\tilde{g}_{k}(\mathbf{x}(t))|$. Thus, in both cases, we have $Q_k(t)  \geq \vert \tilde{g}_{k}(\mathbf{x}(t))\vert$.  Squaring both sides and summing over $k\in\{1,2,\ldots,m\}$ yields $\Vert \mathbf{Q}(t)\Vert^2 \geq \Vert \tilde{\mathbf{g}}(\mathbf{x}(t))\Vert^2$.  
\item Fix $t\in\{0,1,\ldots\}$. Define vector $\mathbf{h}=[h_{1}, \ldots, h_{m}]\tran$ by $h_{k} = |\tilde{g}_{k}(\mathbf{x}(t))|, \forall k\in\{1,2,\ldots,m\}$. Note that $\Vert \mathbf{h}\Vert = \Vert \tilde{\mathbf{g}}(\mathbf{x}(t))\Vert$.  For any $k\in\{1,2,\ldots, m\}$,  by the virtual update equation we have 
\begin{align*}
Q_k(t) &= \max\{-\tilde{g}_{k}(\mathbf{x}(t)), Q_{k}(t-1) + \tilde{g}_{k}(\mathbf{x}(t))\}\\
&\leq |Q_k(t-1)| + |\tilde{g}_{k}(\mathbf{x}(t))| \\ &= Q_k(t-1) + h_k.
\end{align*}
Squaring both sides and summing over $k\in\{1,2,\ldots,m\}$ yields $\Vert \mathbf{Q}(t)\Vert^{2} \leq \Vert \mathbf{Q}(t-1) + \mathbf{h}\Vert^{2}$, which is equivalent to $\Vert \mathbf{Q}(t)\Vert \leq \Vert \mathbf{Q}(t-1) + \mathbf{h}\Vert$. Finally, by the triangle inequality $\Vert \mathbf{Q}(t-1) + \mathbf{h}\Vert \leq \Vert \mathbf{Q}(t-1)\Vert + \Vert\mathbf{h}\Vert$  and recalling that $\Vert \mathbf{h}\Vert = \Vert \tilde{\mathbf{g}}(\mathbf{x}(t))\Vert$, we have $\Vert \mathbf{Q}(t)\Vert \leq \Vert \mathbf{Q}(t-1)\Vert  + \Vert \tilde{\mathbf{g}}(\mathbf{x}(t))\Vert$.
\end{enumerate}
\end{proof}
\begin{Lem}\label{lm:queue-constraint-inequality}
Let $\mathbf{Q}(t), t\in\{0,1,\ldots\}$ be the sequence generated by Algorithm \ref{alg:new-alg}.  For any $T\geq 1$, we have
\begin{align*}
\sum_{t=1}^{T} g_{k}(\mathbf{x}(t)) \leq  \frac{1}{\gamma}Q_{k}(T), \forall k\in\{1,2,\ldots,m\}.
\end{align*}
\end{Lem}
\begin{proof}
Fix $k\in\{1,2,\ldots,m\}$ and $T \geq 1$.  For any $t \in \{1, 2,\ldots, T\}$ the update rule of Algorithm \ref{alg:new-alg} gives: 
\begin{align*}
Q_k(t) &= \max\{-\tilde{g}_{k}(\mathbf{x}(t)), Q_k(t-1)+\tilde{g}_{k}(\mathbf{x}(t))\} \geq Q_k(t-1) + \tilde{g}_{k}(\mathbf{x}(t)). 
\end{align*}
Hence, $ \tilde{g}_{k}(\mathbf{x}(t)) \leq Q_k(t) - Q_k(t-1)$.  Summing over $t \in \{0, \ldots, T-1\}$ yields 
\begin{align*}
\sum_{t=1}^{T} \tilde{g}_{k}(\mathbf{x}(t)) \leq Q_{k}(T) - Q_{k}(0) \overset{(a)}{\leq} Q_{k}(T).
\end{align*} 
where (a) follows from the initialization rule $Q_k(0)= 0$.  This lemma follows by recalling that $\tilde{g}_{k}(\mathbf{x}) = \gamma g_k(\mathbf{x})$.
\end{proof}

Let $\mathbf{Q}(t) = \big[ Q_1(t), \ldots, Q_m(t)\big]\tran$ be the vector of virtual queue backlogs.  Define  $L(t) = \frac{1}{2} \Vert \mathbf{Q}(t)\Vert^2$. The function $L(t)$ shall be called a \emph{Lyapunov function}. Define the {Lyapunov drift} as 
\begin{align}
\Delta (t) = L(t+1) - L(t) = \frac{1}{2} [ \Vert \mathbf{Q}(t+1)\Vert^{2} - \Vert \mathbf{Q}(t)\Vert^{2}]. \label{eq:def-drift}
\end{align}

The Lyapunov drift in \eqref{eq:def-drift} was originally used in the drift-plus-penalty technique for stochastic optimization in dynamic queueing networks \citep{book_Neely10} and was recently used in \citet{YuNeely17SIOPT} to develop fast $O(1/T)$ converging Lagrangian methods for deterministic optimization by controlling constraint violations through a virtul queue dynamic. While online convex optimization with long term constraints differs from both scenarios, it turns out the Lyapunov drift is still quite useful to jointly analyze regret and constraint violations.

\begin{Lem}\label{lm:drift}At each round $t\in\{0,1,2,\ldots\}$ in Algorithm \ref{alg:new-alg}, an upper bound of the Lyapunov drift is given by
\begin{align}
\Delta(t) \leq [\mathbf{Q}(t)]\tran \tilde{\mathbf{g}}(\mathbf{x}(t+1))  +\Vert \tilde{\mathbf{g}}(\mathbf{x}(t+1))\Vert^2.  \label{eq:drift}
\end{align}
\end{Lem}

\begin{proof}~
Fix $t\in\{0,1,\ldots,\}$. The virtual queue update equations $Q_k(t+1) = \max\{-\tilde{g}_{k}(\mathbf{x}(t+1)), Q_k(t) + \tilde{g}_{k}(\mathbf{x}(t+1))\}, \forall k\in \{1,2,\ldots,m\}$ can be rewritten as
\begin{align}
Q_k(t+1) = Q_k(t) + h_k(\mathbf{x}(t+1)), \forall k\in \{1,2,\ldots,m\}, \label{eq:modified-virtual-queue}
\end{align} 
where 
\begin{equation*}
h_k(\mathbf{x}(t)) = \left \{ \begin{array}{cl}  \tilde{g}_{k}(\mathbf{x}(t+1)), & \text{if}~Q_k(t) + \tilde{g}_{k}(\mathbf{x}(t+1)) \geq -\tilde{g}_{k}(\mathbf{x}(t+1))\\
-Q_k(t) - \tilde{g}_{k}(\mathbf{x}(t+1)),  & \text{else} \end{array} \right. \forall k.
\end{equation*}
Fix $k\in\{1,2,\ldots, m\}$. Squaring both sides of \eqref{eq:modified-virtual-queue} and dividing by $2$ yields: 
\begin{align*}
&\frac{1}{2}[Q_k (t+1) ]^2 \\
= &\frac{1}{2}[Q_k(t)]^2 + \frac{1}{2}[h_k(\mathbf{x}(t+1))]^2 +  Q_k (t) h_k(\mathbf{x}(t+1)) \\
= &\frac{1}{2}[Q_k(t)]^2 + \frac{1}{2}[h_k(\mathbf{x}(t+1))]^2 +  Q_k (t) \tilde{g}_{k}(\mathbf{x}(t+1)) +  Q_k (t)[h_k(\mathbf{x}(t+1)) -\tilde{g}_{k}(\mathbf{x}(t+1))]  \\
\overset{(a)}{=}& \frac{1}{2}[Q_k(t)]^2 + \frac{1}{2}[h_k(\mathbf{x}(t+1))]^2 + Q_k (t) \tilde{g}_{k}(\mathbf{x}(t+1)) \\&\quad - [ h_k (\mathbf{x}(t+1)) + \tilde{g}_{k}(\mathbf{x}(t+1))] [ h_k(\mathbf{x}(t+1)) -\tilde{g}_{k}(\mathbf{x}(t+1)) ]\\
=& \frac{1}{2}[Q_k(t)]^2 - \frac{1}{2}[h_k(\mathbf{x}(t+1))]^2 + Q_k (t) \tilde{g}_{k}(\mathbf{x}(t+1))  + [\tilde{g}_{k}(\mathbf{x}(t+1))]^2 \\
\leq &\frac{1}{2}[Q_k(t)]^2 + Q_k (t) \tilde{g}_{k}(\mathbf{x}(t+1))  + [\tilde{g}_{k}(\mathbf{x}(t+1))]^2,
\end{align*}
where $(a)$ follows from the fact that $Q_k (t) [ h_k (\mathbf{x}(t+1)) -\tilde{g}_{k}(\mathbf{x}(t+1)) ] = - [h_k (\mathbf{x}(t+1)) +\tilde{g}_{k}(\mathbf{x}(t+1))] \cdot [ h_k (\mathbf{x}(t+1)) -\tilde{g}_{k}(\mathbf{x}(t+1)) ]$, which can be shown by considering $h_k (\mathbf{x}(t+1)) = \tilde{g}_{k}(\mathbf{x}(t+1))$ and $h_k (\mathbf{x}(t+1)) \neq \tilde{g}_{k}(\mathbf{x}(t+1))$.
Summing over $k\in\{1,2,\ldots,m\}$ yields 
\begin{align*}
\frac{1}{2}\Vert\mathbf{Q}(t+1)\Vert^{2} \leq \frac{1}{2}\Vert\mathbf{Q}(t)\Vert^{2}  + [\mathbf{Q}(t)]\tran \tilde{\mathbf{g}}(\mathbf{x}(t+1)) + \Vert \tilde{\mathbf{g}}(\mathbf{x}(t+1))\Vert^2. 
\end{align*}
Rearranging the terms yields the desired result.
\end{proof}

\section{Regret and Constraint Violation Analysis of Algorithm \ref{alg:new-alg} } \label{sec:analysis}
This section analyzes the regret and constraint violations of Algorithm \ref{alg:new-alg} for online convex optimization with long term constraints under Assumptions \ref{as:basic}-\ref{as:interior-point}.

\subsection{An Upper Bound of the Drift-Plus-Penalty Expression}
\begin{Def}[Strongly Convex Functions]
 Let $\mathcal{X} \subseteq \mathbb{R}^n$ be a convex set. Function $h$ is said to be strongly convex on $\mathcal{X}$ with modulus $\alpha$ if there exists a constant $\alpha>0$ such that $h(\mathbf{x}) - \frac{1}{2} \alpha \Vert \mathbf{x} \Vert^2$ is convex on $\mathcal{X}$.
\end{Def}

By the definition of strongly convex functions, if $h(\mathbf{x})$ is convex and $\alpha>0$, then $h(\mathbf{x}) + \alpha \Vert \mathbf{x} - \mathbf{x}_0\Vert^2$ is strongly convex with modulus $2\alpha$ for any constant $\mathbf{x}_0$.  The following lemma summarizes a useful property of the minimizer for a strongly convex function.
\begin{Lem} [Corollary 1 in \citet{YuNeely17SIOPT}] \label{lm:strong-convex-quadratic-optimality}
Let $\mathcal{X} \subseteq \mathbb{R}^{n}$ be a convex set. Let function $h$ be strongly convex on $\mathcal{X}$ with modulus $\alpha$ and $\mathbf{x}^{opt}$ be a global minimum of $h$ on $\mathcal{X}$. Then, $h(\mathbf{x}^{opt}) \leq h(\mathbf{x}) - \frac{\alpha}{2} \Vert \mathbf{x}^{opt} - \mathbf{x}\Vert^{2}$ for all $\mathbf{x}\in \mathcal{X}$.
\end{Lem}

\begin{Lem}\label{lm:dpp-bound}
Consider online convex optimization with long term constraints under Assumption \ref{as:basic}. Let $\mathbf{x}^{\ast}\in \mathcal{X}_{0}$ be any fixed solution that satisfies $\mathbf{g}(\mathbf{x}^{\ast})\leq \mathbf{0}$, e.g., $\mathbf{x}^{\ast} = \argmin_{\mathbf{x}\in \mathcal{X}} \sum_{t=1}^{T} f^{t}(\mathbf{x})$.  Let $\gamma>0$ and $\eta>0$ be arbitrary. If $\alpha \geq  \frac{1}{2}[\gamma^2\beta^{2} + \eta]$ in Algorithm \ref{alg:new-alg}, then for all $t\geq 1$, we have
\begin{align*}
&\Delta(t) + f^{t}(\mathbf{x}(t)) \\
\leq &f^{t}(\mathbf{x}^{\ast}) + \alpha [\Vert \mathbf{x} ^{\ast}- \mathbf{x}(t)\Vert^{2} - \Vert \mathbf{x}^{\ast} - \mathbf{x}(t+1)\Vert^{2}]  + \frac{1}{2} [ \Vert \tilde{\mathbf{g}}(\mathbf{x}(t+1))\Vert^{2} - \Vert \tilde{\mathbf{g}}(\mathbf{x}(t))\Vert^{2}] + \frac{1}{2\eta} D^{2},
\end{align*}    
where $\beta$ and $D$ are defined in Assumption \ref{as:basic}.
\end{Lem}

\begin{proof}
Fix $t\geq 1$.  Note that part 2 of Lemma \ref{lm:virtual-queue} implies that $\mathbf{Q}(t) + \tilde{\mathbf{g}}(\mathbf{x}(t))$ is component-wise nonnegative. Hence, $[\nabla f^{t}(\mathbf{x}(t))]\tran [\mathbf{x} -\mathbf{x}(t)] + [\mathbf{Q}(t) + \tilde{\mathbf{g}}(\mathbf{x}(t))]\tran \tilde{\mathbf{g}}(\mathbf{x})$ is a convex function with respect to $\mathbf{x}$. Since $\alpha\Vert \mathbf{x}-\mathbf{x}(t)\Vert^{2}$ is strongly convex with respect to $\mathbf{x}$ with modulus $2\alpha$, it follows that 
\begin{align*}
[\nabla f^{t}(\mathbf{x}(t))]\tran [\mathbf{x} -\mathbf{x}(t)]  + [\mathbf{Q}(t) + \tilde{\mathbf{g}}(\mathbf{x}(t))]\tran \tilde{\mathbf{g}}(\mathbf{x}) +  \alpha\Vert \mathbf{x} -\mathbf{x}(t) \Vert^{2}
\end{align*}
is strongly convex with respect to $\mathbf{x}$ with modulus $2\alpha$.  

Since $\mathbf{x}(t+1)$ is chosen to minimize the above strongly convex function, by Lemma \ref{lm:strong-convex-quadratic-optimality}, we have
\begin{align*}
&[\nabla f^{t}(\mathbf{x}(t))]\tran [\mathbf{x}(t+1) -\mathbf{x}(t)]  + [\mathbf{Q}(t) + \tilde{\mathbf{g}}(\mathbf{x}(t))]\tran \tilde{\mathbf{g}}(\mathbf{x}(t+1)) +  \alpha\Vert \mathbf{x}(t+1) -\mathbf{x}(t) \Vert^{2}\\
\leq &[\nabla f^{t}(\mathbf{x}(t))]\tran [\mathbf{x}^{\ast} -\mathbf{x}(t)]  + [\mathbf{Q}(t) + \tilde{\mathbf{g}}(\mathbf{x}(t))]\tran \tilde{\mathbf{g}}(\mathbf{x}^{\ast}) +  \alpha\Vert \mathbf{x}^{\ast} -\mathbf{x}(t) \Vert^{2}- \alpha \Vert \mathbf{x}^{\ast} - \mathbf{x}(t+1)\Vert^{2}.
\end{align*}

Adding $f^{t}(\mathbf{x}(t))$ on both sides yields
\begin{align*}
&f^{t}(\mathbf{x}(t))+ [\nabla f^{t}(\mathbf{x}(t))]\tran [\mathbf{x}(t+1) -\mathbf{x}(t)]  + [\mathbf{Q}(t) + \tilde{\mathbf{g}}(\mathbf{x}(t))]\tran \tilde{\mathbf{g}}(\mathbf{x}(t+1))  \\ & +  \alpha\Vert \mathbf{x}(t+1) -\mathbf{x}(t) \Vert^{2}\\
\leq &f^{t}(\mathbf{x}(t))+ [\nabla f^{t}(\mathbf{x}(t))]\tran [\mathbf{x}^{\ast} -\mathbf{x}(t)]  + [\mathbf{Q}(t) + \tilde{\mathbf{g}}(\mathbf{x}(t))]\tran \tilde{\mathbf{g}}(\mathbf{x}^{\ast}) +  \alpha\Vert \mathbf{x}^{\ast} -\mathbf{x}(t) \Vert^{2} \\ &- \alpha \Vert \mathbf{x}^{\ast} - \mathbf{x}(t+1)\Vert^{2}\\
\overset{(a)}{\leq}&f^{t}(\mathbf{x}^{\ast}) +  \underbrace{[\mathbf{Q}(t) + \tilde{\mathbf{g}}(\mathbf{x}(t))]\tran \tilde{\mathbf{g}}(\mathbf{x}^{\ast})}_{\leq 0}+ \alpha [\Vert \mathbf{x}^{\ast} - \mathbf{x}(t)\Vert^{2}- \Vert \mathbf{x}^{\ast} - \mathbf{x}(t+1)\Vert^{2}]\\
\overset{(b)}{\leq}&f^{t}(\mathbf{x}^{\ast}) + \alpha [\Vert \mathbf{x}^{\ast} - \mathbf{x}(t)\Vert^{2}-  \Vert \mathbf{x}^{\ast} - \mathbf{x}(t+1)\Vert^{2}],
\end{align*}
where (a) follows from the convexity of function $f^{t}(\mathbf{x})$; and (b) follows by using the fact that $\tilde{g}_{k}(\mathbf{x}^{\ast})\leq 0$ and $Q_{k}(t) + \tilde{g}_{k}(\mathbf{x}(t))\geq 0$ (i.e., part 2 in Lemma \ref{lm:virtual-queue}) for all $k\in\{1,2,\dots,m\}$ to eliminate the term marked by an underbrace.  

Rearranging terms yields
\begin{align}
&f^{t}(\mathbf{x}(t))+ [\mathbf{Q}(t)]\tran \tilde{\mathbf{g}}(\mathbf{x}(t+1))\nonumber \\
\leq & f^{t}(\mathbf{x}^{\ast}) + \alpha [\Vert \mathbf{x}^{\ast} - \mathbf{x}(t)\Vert^{2}-  \Vert \mathbf{x}^{\ast} - \mathbf{x}(t+1)\Vert^{2}]- \alpha \Vert \mathbf{x}(t+1) - \mathbf{x}(t)\Vert^{2} \nonumber \\ &-[\nabla f^{t}(\mathbf{x}(t))]\tran [\mathbf{x}(t+1) -\mathbf{x}(t)] -[\tilde{\mathbf{g}}(\mathbf{x}(t))]\tran \tilde{\mathbf{g}}(\mathbf{x}(t+1)). \label{eq:pf-alg2-dpp-bound-eq1}
\end{align}

For any $\eta>0$, we have
\begin{align}
-[\nabla f^{t}(\mathbf{x}(t))]\tran [\mathbf{x}(t+1) -\mathbf{x}(t)] \overset{(a)}{\leq}& \Vert \nabla f^{t}(\mathbf{x}(t))\Vert \Vert \mathbf{x}(t+1) -\mathbf{x}(t)\Vert\nonumber\\
=&[\frac{1}{\sqrt{\eta}} \Vert \nabla f^{t}(\mathbf{x}(t))\Vert ] [\sqrt{\eta} \Vert \mathbf{x}(t+1) -\mathbf{x}(t)\Vert] \nonumber\\
\overset{(b)}{\leq}& \frac{1}{2\eta} \Vert \nabla f^{t}(\mathbf{x}(t))\Vert^{2} + \frac{1}{2}\eta \Vert \mathbf{x}(t+1) -\mathbf{x}(t)\Vert^{2} \nonumber \\
\overset{(c)}{\leq}&\frac{1}{2\eta} D^{2} + \frac{1}{2}\eta \Vert \mathbf{x}(t+1) -\mathbf{x}(t)\Vert^{2}, \label{eq:pf-alg2-dpp-bound-eq2}
\end{align}
where (a) follows from the Cauchy-Schwarz inequality; (b) follows from the basic inequality $ab \leq \frac{1}{2}(a^{2} + b^{2}), \forall a,b\in \mathbb{R} $; and (c) follows from Assumption \ref{as:basic}.

Note that $\mathbf{u}_{1}\tran \mathbf{u}_{2} = \frac{1}{2}[ \Vert \mathbf{u}_{1}\Vert^{2} + \Vert\mathbf{u}_{2}\Vert^2 - \Vert \mathbf{u}_{1} - \mathbf{u}_{2}\Vert^{2}]$ for any $\mathbf{u}_{1},\mathbf{u}_{2}\in \mathbb{R}^{m}$. Thus, we have
\begin{align}
[\tilde{\mathbf{g}}(\mathbf{x}(t))]\tran \tilde{\mathbf{g}}(\mathbf{x}(t+1)) = \frac{1}{2}\big[ \Vert \tilde{\mathbf{g}}(\mathbf{x}(t))\Vert^{2} + \Vert\tilde{\mathbf{g}}(\mathbf{x}(t+1))\Vert^{2} - \Vert \tilde{\mathbf{g}}(\mathbf{x}(t+1)) - \tilde{\mathbf{g}}(\mathbf{x}(t))\Vert^{2}\big].\label{eq:pf-alg2-dpp-bound-eq3}
\end{align}

Substituting \eqref{eq:pf-alg2-dpp-bound-eq2} and \eqref{eq:pf-alg2-dpp-bound-eq3} into \eqref{eq:pf-alg2-dpp-bound-eq1} yields
\begin{align}
&f^{t}(\mathbf{x}(t))+ [\mathbf{Q}(t)]\tran \tilde{\mathbf{g}}(\mathbf{x}(t+1))\nonumber \\
\leq & f^{t}(\mathbf{x}^{\ast}) + \alpha [\Vert \mathbf{x}^{\ast} - \mathbf{x}(t)\Vert^{2}-  \Vert \mathbf{x}^{\ast} - \mathbf{x}(t+1)\Vert^{2}] + [\frac{1}{2}\eta-\alpha] \Vert \mathbf{x}(t+1) - \mathbf{x}(t)\Vert^{2} + \frac{1}{2\eta} D^{2} \nonumber \\ & +\frac{1}{2}\Vert \tilde{\mathbf{g}}(\mathbf{x}(t+1)) - \tilde{\mathbf{g}}(\mathbf{x}(t))\Vert^{2} -\frac{1}{2}\Vert \tilde{\mathbf{g}}(\mathbf{x}(t))\Vert^{2} -\frac{1}{2}\Vert\tilde{\mathbf{g}}(\mathbf{x}(t+1))\Vert^{2} \nonumber\\
\overset{(a)}{\leq} & f^{t}(\mathbf{x}^{\ast}) + \alpha [\Vert \mathbf{x}^{\ast} - \mathbf{x}(t)\Vert^{2}-  \Vert \mathbf{x}^{\ast} - \mathbf{x}(t+1)\Vert^{2}] + [\frac{1}{2}\gamma^2\beta^{2} + \frac{1}{2}\eta- \alpha] \Vert \mathbf{x}(t+1) - \mathbf{x}(t)\Vert^{2}  \nonumber \\ & + \frac{1}{2\eta} D^{2}   -\frac{1}{2}\Vert \tilde{\mathbf{g}}(\mathbf{x}(t))\Vert^{2} -\frac{1}{2}\Vert\tilde{\mathbf{g}}(\mathbf{x}(t+1))\Vert^{2} \nonumber\\
\overset{(b)}{\leq} & f^{t}(\mathbf{x}^{\ast}) + \alpha [\Vert \mathbf{x}^{\ast} - \mathbf{x}(t)\Vert^{2}-  \Vert \mathbf{x}^{\ast} - \mathbf{x}(t+1)\Vert^{2}] + \frac{1}{2\eta} D^{2} -\frac{1}{2}\Vert \tilde{\mathbf{g}}(\mathbf{x}(t))\Vert^{2} -\frac{1}{2}\Vert\tilde{\mathbf{g}}(\mathbf{x}(t+1))\Vert^{2}, \label{eq:pf-alg2-dpp-bound-eq4}
\end{align}
where (a) follows from the fact that $\Vert \tilde{\mathbf{g}}(\mathbf{x}(t+1)) - \tilde{\mathbf{g}}(\mathbf{x}(t))\Vert \leq \gamma \beta \Vert \mathbf{x}(t+1) - \mathbf{x}(t)\Vert$, which further follows from Lemma \ref{lm:scaled-constant}; and (b) follows from the fact that $\alpha \geq \frac{1}{2}[\gamma^2\beta^{2} + \eta]$.

By Lemma \ref{lm:drift}, we have 
\begin{align}
\Delta(t) \leq [\mathbf{Q}(t)]\tran \tilde{\mathbf{g}}(\mathbf{x}(t+1)) +  \Vert\tilde{\mathbf{g}}(\mathbf{x}(t+1))\Vert^{2}. \label{eq:pf-alg2-dpp-bound-eq5}
\end{align}
Summing \eqref{eq:pf-alg2-dpp-bound-eq4} and \eqref{eq:pf-alg2-dpp-bound-eq5} together yields
\begin{align*}
&\Delta(t) + f^{t}(\mathbf{x}(t)) \\
\leq & f^{t}(\mathbf{x}^{\ast}) +  \alpha [\Vert \mathbf{x} ^{\ast}- \mathbf{x}(t)\Vert^{2} - \Vert \mathbf{x}^{\ast} - \mathbf{x}(t+1)\Vert^{2}]  + \frac{1}{2} [ \Vert \tilde{\mathbf{g}}(\mathbf{x}(t+1))\Vert^{2} - \Vert \tilde{\mathbf{g}}(\mathbf{x}(t))\Vert^{2}] + \frac{1}{2\eta} D^{2}.
\end{align*}
\end{proof}

\subsection{Regret Analysis}
\begin{Thm} \label{thm:regret-bound}
Consider online convex optimization with long term constraints under Assumption \ref{as:basic}. Let $\mathbf{x}^{\ast}\in \mathcal{X}_{0}$ be any fixed solution that satisfies $\mathbf{g}(\mathbf{x}^{\ast})\leq \mathbf{0}$, e.g., $\mathbf{x}^{\ast} = \argmin_{\mathbf{x}\in \mathcal{X}} \sum_{t=1}^{T} f^{t}(\mathbf{x})$.  
\begin{enumerate}
\item  Let $\gamma>0$ and $\eta>0$ be arbitrary. If $\alpha \geq  \frac{1}{2}[\gamma^2\beta^{2} + \eta]$ in Algorithm \ref{alg:new-alg}, then for all $T\geq 1$, we have
\begin{align*}
\sum_{t=1}^{T}f^{t}(\mathbf{x}(t)) \leq \sum_{t=1}^{T}f^{t}(\mathbf{x}^{\ast}) + \alpha \Vert \mathbf{x}^{\ast} - \mathbf{x}(1)\Vert^{2} + \frac{1}{2\eta} D^{2}T. 
\end{align*}
\item If $\gamma = T^{1/4}, \eta =\sqrt{T}$ and $\alpha= \frac{1}{2}[\beta^{2} +1] \sqrt{T}$ in Algorithm \ref{alg:new-alg}, then for all $T\geq 1$, we have
\begin{align*}
\sum_{t=1}^{T}f^{t}(\mathbf{x}(t)) \leq \sum_{t=1}^{T}f^{t}(\mathbf{x}^{\ast}) +  O(\sqrt{T}).
\end{align*}

\end{enumerate}

\end{Thm}

\begin{proof}
Fix $T\geq 1$. Since $\alpha \geq  \frac{1}{2}[\gamma^2\beta^{2} + \eta]$, by Lemma \ref{lm:dpp-bound}, for all $t\in\{1,2,\ldots,T\}$, we have 
\begin{align*}
&\Delta(t) + f^{t}(\mathbf{x}(t)) \\
\leq &f^{t}(\mathbf{x}^{\ast}) + \alpha [\Vert \mathbf{x} ^{\ast}- \mathbf{x}(t)\Vert^{2} - \Vert \mathbf{x}^{\ast} - \mathbf{x}(t+1)\Vert^{2}]  + \frac{1}{2} [ \Vert \tilde{\mathbf{g}}(\mathbf{x}(t+1))\Vert^{2} - \Vert \tilde{\mathbf{g}}(\mathbf{x}(t))\Vert^{2}] + \frac{1}{2\eta} D^{2}.
\end{align*}  

Summing over $t\in\{1,2,\ldots,T\}$ yields
\begin{align*}
\sum_{t=1}^{T}\Delta(t) + \sum_{t=1}^{T}f^{t}(\mathbf{x}(t)) \leq & 
\sum_{t=1}^{T}f^{t}(\mathbf{x}^{\ast}) +  \alpha \sum_{t=1}^{T}[\Vert \mathbf{x} ^{\ast}- \mathbf{x}(t)\Vert^{2} - \Vert \mathbf{x}^{\ast} - \mathbf{x}(t+1)\Vert^{2}]  \\ &+ 
\frac{1}{2}\sum_{t=1}^{T} [ \Vert \tilde{\mathbf{g}}(\mathbf{x}(t+1))\Vert^{2} - \Vert \tilde{\mathbf{g}}(\mathbf{x}(t))\Vert^{2}] + \frac{1}{2\eta} D^{2}T.
\end{align*}

Recalling that $\Delta(t) = L(t+1) - L(t)$ and simplifying summations yields
{\small
\begin{align*}
&L(T+1) - L(1) +  \sum_{t=1}^{T}f^{t}(\mathbf{x}(t))\\
\leq& \sum_{t=1}^{T}f^{t}(\mathbf{x}^{\ast}) + \alpha \Vert \mathbf{x}^{\ast} - \mathbf{x}(1)\Vert^{2} -  \alpha \Vert \mathbf{x}^{\ast} - \mathbf{x}(T+1)\Vert^{2} + \frac{1}{2} \Vert \tilde{\mathbf{g}}(\mathbf{x}(T+1))\Vert^{2} - \frac{1}{2} \Vert \tilde{\mathbf{g}}(\mathbf{x}(1))\Vert^{2} + \frac{1}{2\eta} D^{2}T\\
\leq & \sum_{t=1}^{T}f^{t}(\mathbf{x}^{\ast}) + \alpha \Vert \mathbf{x}^{\ast} - \mathbf{x}(1)\Vert^{2} + \frac{1}{2} \Vert \tilde{\mathbf{g}}(\mathbf{x}(T+1))\Vert^{2}  - \frac{1}{2} \Vert \tilde{\mathbf{g}}(\mathbf{x}(1))\Vert^{2}+ \frac{1}{2\eta} D^{2}T. 
\end{align*}
}
Rearranging terms yields
\begin{align*}
&\sum_{t=1}^{T}f^{t}(\mathbf{x}(t))\\
\leq &\sum_{t=1}^{T}f^{t}(\mathbf{x}^{\ast}) + \alpha \Vert \mathbf{x}^{\ast} - \mathbf{x}(1)\Vert^{2} + \frac{1}{2} \Vert \tilde{\mathbf{g}}(\mathbf{x}(T+1))\Vert^{2} + L(1) - L(T+1) + \frac{1}{2\eta} D^{2}T\\
\overset{(a)}{=} &\sum_{t=1}^{T}f^{t}(\mathbf{x}^{\ast}) +  \alpha \Vert \mathbf{x}^{\ast} - \mathbf{x}(1)\Vert^{2} + \frac{1}{2} \Vert \tilde{\mathbf{g}}(\mathbf{x}(T+1))\Vert^{2}  - \frac{1}{2} \Vert \tilde{\mathbf{g}}(\mathbf{x}(1))\Vert^{2}+ \frac{1}{2} \Vert \mathbf{Q}(1)\Vert^{2} \\ &- \frac{1}{2}\Vert \mathbf{Q}(T+1)\Vert^{2} + \frac{1}{2\eta} D^{2}T\\
\overset{(b)}{\leq}&\sum_{t=1}^{T}f^{t}(\mathbf{x}^{\ast}) +  \alpha \Vert \mathbf{x}^{\ast} - \mathbf{x}(1)\Vert^{2}  + \frac{1}{2\eta} D^{2}T
\end{align*}
where (a) follows form the definition that $L(1) = \frac{1}{2} \Vert \mathbf{Q}(1)\Vert^{2}$ and $L(T+1) = \frac{1}{2}\Vert \mathbf{Q}(T+1)\Vert^{2}$; and (b) follows from the fact that $\Vert \mathbf{Q}(1)\Vert^{2} = \Vert \tilde{\mathbf{g}}(\mathbf{x}(1))\Vert^{2}$ and $\Vert \mathbf{Q}(T+1)\Vert^{2} \geq \Vert \tilde{\mathbf{g}}(\mathbf{x}(T+1))\Vert^{2}$, i.e., part 3 in Lemma \ref{lm:virtual-queue}.

Thus, the first part of this theorem follows. Note that if we let $\gamma = T^{1/4}$ and $\eta = \sqrt{T}$, then $\alpha = \frac{1}{2}[\beta^2+1]\sqrt{T} \geq \frac{1}{2}[\gamma^2\beta^{2} + \eta]$. The second part of this theorem follows by substituting $\gamma = T^{1/4}, \eta =\sqrt{T}$ and $\alpha = \frac{1}{2}[\beta^2+1]\sqrt{T}$ into the first part of this theorem.  Thus, we have 
\begin{align*}
\sum_{t=1}^{T}f^{t}(\mathbf{x}(t)) \leq &\sum_{t=1}^{T}f^{t}(\mathbf{x}^{\ast}) + \frac{1}{2}[\beta^2+1] \Vert \mathbf{x}^{\ast} - \mathbf{x}(1)\Vert^{2} \sqrt{T}  + \frac{1}{2} D^{2}\sqrt{T} \\
=&\sum_{t=1}^{T}f^{t}(\mathbf{x}^{\ast}) +  O(\sqrt{T}).
\end{align*}
\end{proof}
\subsection{An Upper Bound of the Virtual Queue Vector}
It remains to establish a bound on constraint violations.  Lemma \ref{lm:queue-constraint-inequality} shows this can be done by bounding $\Vert \mathbf{Q}(t)\Vert$.
 
\begin{Lem}\label{lm:queue-decrease-condition}
Consider online convex optimization with long term constraints under Assumptions \ref{as:basic}-\ref{as:interior-point}.  At each round $t\in\{1,2,\ldots,\}$ in Algorithm \ref{alg:new-alg}, if $\Vert \mathbf{Q}(t)\Vert  > \gamma G + \frac{\alpha R^{2} + DR + 2\gamma^2G^{2}}{\gamma\epsilon}$, where $D, G$ and $R$ are defined in Assumption \ref{as:basic} and $\epsilon$ is defined in Assumption \ref{as:interior-point}, then $\Vert \mathbf{Q}(t+1)\Vert < \Vert \mathbf{Q}(t)\Vert$.
\end{Lem}
\begin{proof}
Let $\hat{\mathbf{x}}\in \mathcal{X}_{0}$ and $\epsilon>0$ be defined in Assumption \ref{as:interior-point}. Fix $t\geq 0$.  Since $\mathbf{x}(t+1)$ is chosen to minimize $ [\nabla f^{t}(\mathbf{x}(t))]\tran[\mathbf{x} -\mathbf{x}(t)] + [\mathbf{Q}(t) + \tilde{\mathbf{g}}(\mathbf{x}(t))]\tran \tilde{\mathbf{g}}(\mathbf{x}) + \alpha \Vert \mathbf{x} - \mathbf{x}(t)\Vert^{2}$, we have
\begin{align*}
& [\nabla f^{t}(\mathbf{x}(t))]\tran[\mathbf{x}(t+1) -\mathbf{x}(t)] + [\mathbf{Q}(t) + \tilde{\mathbf{g}}(\mathbf{x}(t))]\tran \tilde{\mathbf{g}}(\mathbf{x}(t+1)) + \alpha \Vert \mathbf{x}(t+1) - \mathbf{x}(t)\Vert^{2} \\
\leq & [\nabla f^{t}(\mathbf{x}(t))]\tran[\hat{\mathbf{x}} -\mathbf{x}(t)] + [\mathbf{Q}(t) + \tilde{\mathbf{g}}(\mathbf{x}(t))]\tran \tilde{\mathbf{g}}(\hat{\mathbf{x}}) + \alpha \Vert \hat{\mathbf{x}} - \mathbf{x}(t)\Vert^{2}\\
\overset{(a)}{\leq} & [\nabla f^{t}(\mathbf{x}(t))]\tran[\hat{\mathbf{x}} -\mathbf{x}(t)] - \gamma \epsilon \sum_{k=1}^{m} [Q_{k}(t) + \tilde{g}_{k}(\mathbf{x}(t))] + \alpha \Vert \hat{\mathbf{x}} - \mathbf{x}(t)\Vert^{2}\\
\overset{(b)}{\leq}& [\nabla f^{t}(\mathbf{x}(t))]\tran[\hat{\mathbf{x}} -\mathbf{x}(t)]  - \gamma \epsilon \Vert \mathbf{Q}(t) + \tilde{\mathbf{g}}(\mathbf{x}(t))\Vert + \alpha \Vert \hat{\mathbf{x}} - \mathbf{x}(t)\Vert^{2} \\
\overset{(c)}{\leq}&  [\nabla f^{t}(\mathbf{x}(t))]\tran [\hat{\mathbf{x}} -\mathbf{x}(t)] - \gamma\epsilon \big[\Vert \mathbf{Q}(t) \Vert -\Vert\tilde{\mathbf{g}}(\mathbf{x}(t))\Vert\big] + \alpha \Vert \hat{\mathbf{x}} - \mathbf{x}(t)\Vert^{2},
\end{align*}
where (a) follows from the fact that $\tilde{g}_{k}(\hat{\mathbf{x}})\leq -\gamma \epsilon, \forall k\in\{1,2,\ldots,m\}$, i.e., Lemma \ref{lm:scaled-constant} and the fact that $Q_{k}(t) + \tilde{g}_{k}(\mathbf{x}(t))\geq 0, \forall k\in\{1,2,\ldots,m\}$, i.e., part 2 in Lemma \ref{lm:virtual-queue}; (b) follows from the basic inequality $\sum_{i=1}^{m} a_{i} \geq \sqrt{\sum_{i=1}^{m} a_{i}^{2}}$ for any nonnegative vector $\mathbf{a}\geq \mathbf{0}$; and (c) follows from the triangle inequality $\Vert \mathbf{x} -\mathbf{y}\Vert \geq \Vert \mathbf{x}\Vert - \Vert \mathbf{y}\Vert, \forall \mathbf{x}, \mathbf{y}\in \mathbb{R}^{m}$.

Rearranging terms yields
\begin{align}
&[\mathbf{Q}(t)]\tran \tilde{\mathbf{g}}(\mathbf{x}(t+1)) \nonumber\\
\leq & -\gamma\epsilon \Vert \mathbf{Q}(t)\Vert + \gamma\epsilon \Vert \tilde{\mathbf{g}}(\mathbf{x}(t))\Vert +  \alpha \Vert \hat{\mathbf{x}} - \mathbf{x}(t)\Vert^{2} -  \alpha \Vert \mathbf{x}(t+1) - \mathbf{x}(t)\Vert^{2}  \nonumber \\
&+ [\nabla f^{t}(\mathbf{x}(t))]\tran[\hat{\mathbf{x}} -\mathbf{x}(t)] -[\nabla f^{t}(\mathbf{x}(t))]\tran[\mathbf{x}(t+1) -\mathbf{x}(t)]  - [\tilde{\mathbf{g}}(\mathbf{x}(t))]\tran \tilde{\mathbf{g}}(\mathbf{x}(t+1)) \nonumber\\
\leq& - \gamma \epsilon \Vert \mathbf{Q}(t)\Vert + \gamma\epsilon \Vert \tilde{\mathbf{g}}(\mathbf{x}(t))\Vert +  \alpha \Vert \hat{\mathbf{x}} - \mathbf{x}(t)\Vert^{2} + [\nabla f^{t}(\mathbf{x}(t))]\tran[\hat{\mathbf{x}} -\mathbf{x}(t+1)] \nonumber \\
& - [\tilde{\mathbf{g}}(\mathbf{x}(t))]\tran \tilde{\mathbf{g}}(\mathbf{x}(t+1)) \nonumber\\
\overset{(a)}{\leq}& -\gamma \epsilon \Vert \mathbf{Q}(t)\Vert + \gamma \epsilon \Vert \tilde{\mathbf{g}}(\mathbf{x}(t))\Vert +  \alpha \Vert \hat{\mathbf{x}} - \mathbf{x}(t)\Vert^{2} + \Vert \nabla f^{t}(\mathbf{x}(t))\Vert \Vert \hat{\mathbf{x}} - \mathbf{x}(t+1)\Vert \nonumber \\ &
 +\Vert \tilde{\mathbf{g}}(\mathbf{x}(t))\Vert \Vert\tilde{\mathbf{g}}(\mathbf{x}(t+1))\Vert\nonumber\\
\overset{(b)}{\leq}&-\gamma\epsilon \Vert \mathbf{Q}(t)\Vert + \gamma^2\epsilon G + \alpha R^{2} + DR + \gamma^2G^{2}, \label{eq:pf-queue-bound-eq1}
\end{align}
where (a) follows from Cauchy-Schwarz inequality and (b) follows from Assumption \ref{as:basic} and Lemma \ref{lm:scaled-constant}.

By Lemma \ref{lm:drift}, we have
\begin{align*}
\Delta(t) \leq& [\mathbf{Q}(t)]\tran \tilde{\mathbf{g}}(\mathbf{x}(t+1)) + \Vert \tilde{\mathbf{g}}(\mathbf{x}(t+1))\Vert^{2} \\
\overset{(a)}\leq &  [\mathbf{Q}(t)]\tran \tilde{\mathbf{g}}(\mathbf{x}(t+1)) + \gamma^2G^{2}\\
\overset{(b)}\leq &- \gamma \epsilon \Vert \mathbf{Q}(t)\Vert + \gamma^2 \epsilon G + \alpha R^{2} + DR +2\gamma^2G^{2},
\end{align*}
where (a) follows from Lemma \ref{lm:scaled-constant} and (b) follows from \eqref{eq:pf-queue-bound-eq1}.

Thus, if $\Vert \mathbf{Q}(t)\Vert >  \gamma G + \frac{\alpha R^{2} + DR + 2\gamma^2G^{2}}{\gamma\epsilon}$, then $\Delta(t) <0$. That is, $\Vert \mathbf{Q}(t+1)\Vert < \Vert \mathbf{Q}(t)\Vert$. 
\end{proof}

\begin{Cor}\label{cor:queue-bound}
Consider online convex optimization with long term constraints under Assumptions \ref{as:basic}-\ref{as:interior-point}.   At each round $t\in\{1,2,\ldots,\}$ in Algorithm \ref{alg:new-alg}, $$\Vert \mathbf{Q}(t)\Vert  \leq 2\gamma G + \frac{\alpha R^{2} + DR + 2\gamma^2G^{2}}{\gamma\epsilon},$$ where $D, G$ and $R$ are defined in Assumption \ref{as:basic} and $\epsilon>0$ is defined in Assumption \ref{as:interior-point}.
\end{Cor}
\begin{proof}

Note that $\Vert \mathbf{Q}(1)\Vert \overset{(a)}{\leq} \Vert \mathbf{Q}(0)\Vert + \Vert \tilde{\mathbf{g}}(\mathbf{x}(0))\Vert \overset{(b)}{\leq} \gamma G \leq 2\gamma G + \frac{\alpha R^{2} + DR + 2\gamma^2G^{2}}{\gamma\epsilon}$, where (a) follows from Lemma \ref{lm:virtual-queue} and (b) follows from Lemma \ref{lm:scaled-constant}. We need to show $\Vert \mathbf{Q}(t)\Vert \leq 2\gamma G + \frac{\alpha R^{2} + DR + 2\gamma^2G^{2}}{\gamma\epsilon}$ for all rounds $t\geq 2$. This can be proven by contradiction as follows:

Assume that $\Vert \mathbf{Q}(t)\Vert > 2\gamma G + \frac{\alpha R^{2} + DR + 2\gamma^2G^{2}}{\gamma\epsilon}$ happens at some round $t\geq 2$. Let $\tau$ be the first (smallest) round index at which this happens, i.e., $\Vert \mathbf{Q}(\tau)\Vert > 2\gamma G + \frac{\alpha R^{2} + DR + 2\gamma^2G^{2}}{\gamma\epsilon}$. Note that $\tau\geq 2$ since we know $\Vert \mathbf{Q}(1)\Vert \leq 2\gamma G + \frac{\alpha R^{2} + DR + 2\gamma^2G^{2}}{\gamma\epsilon}$. The definition of $\tau$ implies that $\Vert \mathbf{Q}(\tau-1)\Vert \leq  2\gamma G + \frac{\alpha R^{2} + DR + 2\gamma^2G^{2}}{\gamma\epsilon}$.   Now consider the value of $\Vert \mathbf{Q}(\tau-1)\Vert$ in two cases.
\begin{itemize}
\item If $\Vert \mathbf{Q}(\tau-1)\Vert > \gamma G+ \frac{\alpha R^{2} + DR + 2\gamma^2G^{2}}{\gamma\epsilon}$, then by Lemma \ref{lm:queue-decrease-condition}, we must have $\Vert \mathbf{Q}(\tau)\Vert < \Vert \mathbf{Q}(\tau-1)\Vert \leq  2\gamma G + \frac{\alpha R^{2} + DR + 2\gamma^2G^{2}}{\gamma\epsilon}$. This contradicts the definition of $\tau$.
\item  If $\Vert \mathbf{Q}(\tau-1)\Vert \leq \gamma G+ \frac{\alpha R^{2} + DR + 2\gamma^2G^{2}}{\gamma\epsilon}$, then by part 4 in Lemma \ref{lm:virtual-queue}, we must have $\Vert \mathbf{Q}(\tau)\Vert \leq \Vert \mathbf{Q}(\tau-1)\Vert + \Vert \tilde{\mathbf{g}}(\mathbf{x}(\tau))\Vert \overset{(a)}{\leq} \gamma G + \frac{\alpha R^{2} + DR + 2\gamma^2G^{2}}{\gamma \epsilon} + \gamma G = 2\gamma G + \frac{\alpha R^{2} + DR + 2\gamma^2G^{2}}{\gamma \epsilon}$, where (a) follows from Lemma \ref{lm:scaled-constant}. This also contradicts the definition of $\tau$.
\end{itemize}
In both cases, we have a contradiction. Thus, $\Vert \mathbf{Q}(t)\Vert \leq 2 \gamma G + \frac{\alpha R^{2} + DR + 2\gamma^2G^{2}}{\gamma\epsilon}$ for all round $t\geq 2$.
\end{proof}
\subsection{Constraint Violation Analysis}
\begin{Thm}\label{thm:constraint-bound}
Consider online convex optimization with long term constraints under Assumptions \ref{as:basic}-\ref{as:interior-point}. Let $D, \beta, G, R$ and $\epsilon$ be defined in Assumptions  \ref{as:basic}-\ref{as:interior-point}. 
\begin{enumerate}
\item For all $T\geq 1$, we have  $$\sum_{t=1}^{T} g_{k}(\mathbf{x}(t)) \leq 2 G + \frac{\alpha R^{2} + DR}{\gamma^2 \epsilon} + \frac{2G^2}{\epsilon}, \forall k\in\{1,2,\ldots,m\}.$$
\item If $\gamma = T^{1/4}$ and $\alpha = \frac{1}{2}[\beta^2 +1] \sqrt{T}$ in Algorithm \ref{alg:new-alg}, then for all $T\geq 1$, we have $$\sum_{t=1}^{T} g_{k}(\mathbf{x}(t)) \leq 2G + \frac{\frac{1}{2}[\beta^2+1]R^2 + 2G^2}{\epsilon} +  \frac{DR}{\epsilon \sqrt{T}}, \forall k\in\{1,2,\ldots,m\}.$$
\end{enumerate}
\end{Thm}
\begin{proof}
Fix $T\geq 1$ and $k\in\{1,2,\ldots,m\}$. By Lemma \ref{lm:queue-constraint-inequality}, we have
\begin{align*}
\sum_{t=1}^{T} g_{k}(\mathbf{x}(t)) \leq&  \frac{1}{\gamma}Q_{k}(T) \leq  \frac{1}{\gamma}\Vert Q_{k}(T)\Vert \overset{(a)}{\leq} \frac{2\gamma}{\gamma}G + \frac{\alpha R^{2} + DR + 2\gamma^2G^{2}}{\gamma^2\epsilon}, 
\end{align*}
where (a) follows from Corollary \ref{cor:queue-bound}. Thus, the first part of this theorem follows.

The second part of this theorem follows by substituting $\gamma =T^{1/4}$ and $\alpha =\frac{1}{2}[\beta^2 + 1] \sqrt{T}$ into the last inequality.
\end{proof}

\subsection{Performance Summary}

Theorem \ref{thm:regret-bound} and Theorem \ref{thm:constraint-bound} together imply that if we choose\footnote{More precisely, to achieve $O(\sqrt{T})$ regret and $O(1)$ constraint violations, it suffices to choose $\gamma = O(T^{1/4})$ and $\alpha = \frac{1}{2} \beta^{2} \gamma^{2} + O(\sqrt{T})$.} $\gamma = T^{1/4}$ and $\alpha = \frac{1}{2}[\beta^2 +1] \sqrt{T}$ in Algorithm \ref{alg:new-alg}, then we can achieve $O(\sqrt{T})$ regret and $O(1)$ constraint violations for online convex optimization with long term constraints under Assumptions \ref{as:basic}-\ref{as:interior-point}.

Parts (1) of Theorems \ref{thm:regret-bound} and \ref{thm:constraint-bound} suggest that our regret and constraint violations do not rely on the $m$, which is the number of constraints. However, the constraint violation bound does depend on $G^{2}$, where $G$ is defined in Assumption \ref{as:bounded-g}. Note that $G^{2}$ in the worst case can grow linearly with respect to $m$.

Note that the $O(1)$ constraint violation bound proven in Theorem \ref{thm:constraint-bound} is in terms of $D, G, R$ and $\epsilon$ defined in Assumptions  \ref{as:basic}-\ref{as:interior-point}.  However,  the implementation of Algorithm \ref{alg:new-alg} only requires the knowledge of $\beta$, which is known to us since the constraint function $\mathbf{g}(\mathbf{x})$ does not change.  In contrast, the algorithms developed in \citet{Mehdavi12JMLR} and \citet{Jenatton16ICML} have parameters that must be chosen based on the knowledge of $D$, which is usually unknown and can be difficult to estimate in an online optimization scenario.

\section{Extensions} \label{sec:extensions}
This section extends the analysis in the previous section by considering intermediate and unknown time horizon $T$ and by relaxing Assumptions \ref{as:bounded-g}-\ref{as:interior-point}.

\subsection{Intermediate Time Horizon $T$}
Note that parts (1) of Theorems \ref{thm:regret-bound} and \ref{thm:constraint-bound} hold for any $T$. For large $T$, choosing $\eta=T^{1/4}$ and $\alpha = \frac{1}{2}[\beta^2+1]\sqrt{T}$ yields the $O(\sqrt{T})$ regret bound and $O(1)$ constraint violations as proven in parts (2) of both theorems.  For intermediate $T$, the constant factor hidden in the $O(\sqrt{T})$ bound can be important and the $O(1)$ constraint violation bound can be relatively large.  If parameters in Assumptions  \ref{as:basic}-\ref{as:interior-point} are known, we can obtain the best regret and constraint violation bounds by choosing $\gamma$ and $\alpha$ as the solution to  the following geometric program\footnote{By dividing the first two constraints by $z$ and dividing the third constraint by $\alpha$ on both sides, this geometric program can be written into the standard from of geometric programs. Geometric programs can be reformulated into convex programs and can be efficiently solved. See \citet{Boyd07GeometricProgramming} for more discussions on geometric programs.}:
\begin{align*}
\min_{\eta, \gamma, \alpha,z} \quad &z \\
\text{s.t.} \quad  &  \alpha R^2 + 2\gamma^2 G^2 + \frac{1}{2\eta} D^2T \leq z, \\
			 &  2 G + \frac{\alpha R^{2} + DR + 2\gamma^2G^{2}}{\gamma^2\epsilon} \leq z,\\
			 &   \frac{1}{2}[\beta^2 \gamma^2 + \eta] \leq \alpha,\\
			 &  \eta, \gamma, \alpha, z  >0.
\end{align*}
In certain applications, we can choose  $\gamma$ and $\alpha$ to minimize the regret bound subject to the constraint violation guarantee by solving the following geometric program:
\begin{align*}
\min_{\eta, \gamma, \alpha} \quad &\alpha R^2 + 2 \gamma^2 G^2+ \frac{1}{2\eta} D^2T \\
\text{s.t.} \quad  &  2G + \frac{\alpha R^{2} + DR + 2\gamma^2G^{2}}{\gamma^2\epsilon} \leq z_0,\\
			 &   \frac{1}{2}[\beta^2 \gamma^2 + \eta] \leq \alpha,\\
			 &  \eta, \gamma, \alpha>0,
\end{align*}
where $z_0>0$ is a constant that specifies the maximum allowed constraint violation. Or alternatively, we can consider the problem of minimizing the constraint violation subject to the regret bound guarantee.

\subsection{Unknown Time Horizon $T$}

To achieve $O(\sqrt{T})$ regret and $O(1)$ constraint violations, the parameters $\gamma$ and $\alpha$ in Algorithm \ref{alg:new-alg} depend on the time horizon $T$. In the case when $T$ is unknown, we can use the classical ``doubling trick" to achieve $O(\sqrt{T})$ regret and $O(\log_2 T)$ constraint violations.

Suppose we have an online convex optimization algorithm $\mathcal{A}$ whose parameters depend on the time horizon.  In the case when the time horizon $T$ is unknown, the general doubling trick \citep{book_PredictionLearningGames, Shalev-Shwartz11FoundationTrends} is described in Algorithm \ref{alg:doubling-trick}. It is known that the doubling trick can preserve the order of algorithm $\mathcal{A}$'s regret bound in the case when the time horizon $T$ is unknown. The next theorem summarizes that by using the ``doubling trick" for Algorithm \ref{alg:new-alg} with unknown time horizon $T$, we can achieve $O(\sqrt{T})$ regret and $O(\log_2 T)$ constraint violations.

\begin{algorithm}
\caption{The Doubling Trick \citep{book_PredictionLearningGames, Shalev-Shwartz11FoundationTrends}}
\label{alg:doubling-trick}
\begin{itemize}
\item Let algorithm $\mathcal{A}$ be an algorithm whose parameters depend on the time horizon. Let $i=1$. 
\item Repeat until we reach the end of the time horizon
\begin{itemize}
\item Run algorithm $\mathcal{A}$ for $2^{i}$ rounds by using $2^{i}$ as the time horizon.
\item Let $i = i+1$.
\end{itemize}
\end{itemize}
\end{algorithm}

\begin{Thm}
If the time horizon $T$ is unknown, then applying Algorithm \ref{alg:new-alg} with the ``doubling trick" can yield $O(\sqrt{T})$ regret and $O(\log_2 T)$ constraint violations.
\end{Thm}
\begin{proof}
Let $T$ be the unknown time horizon. Define each iteration in the doubling trick as a period.  Since the $i$-th period consists of $2^i$ rounds, we have in total $\lceil \log_2 T\rceil$ periods, where $\lceil x\rceil$ denotes the smallest integer no less than $x$.
\begin{enumerate}
\item The proof of $O(\sqrt{T})$ regret is almost identical to the classical proof.  By Theorem \ref{thm:regret-bound}, there exists a constant $C$ such that the regret in the $i$-th period is at most $C\sqrt{2^i}$. Thus, the total regret is at most 
\begin{align*}
\sum_{i=1}^{\lceil \log_2 T\rceil} C \sqrt{2^i} =&  C \frac{\sqrt{2} [1- \sqrt{2}^{\lceil \log_2 T\rceil}]}{1-\sqrt{2}}\\
=&\frac{\sqrt{2}C}{\sqrt{2}-1} [\sqrt{2}^{\lceil \log_2 T\rceil} -1]\\
\leq& \frac{\sqrt{2}C}{\sqrt{2}-1} \sqrt{2}^{1+ \log_2 T}\\
\leq&\frac{2C}{\sqrt{2}-1} \sqrt{T}
\end{align*}
Thus, the regret bound is $O(\sqrt{T})$ when using the ``doubling trick".
\item The proof of $O(\log_2 T)$ constraint violations is simple.  By Theorem \ref{thm:regret-bound}, there exists a constant $C$ such that the constraint violation in the $i$-th period is at most $C$. Since we have $\lceil \log_2 T\rceil$ periods, the total constraint violation is $C\lceil \log_2 T\rceil$.
\end{enumerate}
\end{proof}

 \subsection{Relaxing Assumptions \ref{as:bounded-set}-\ref{as:interior-point}}

Note that previous works \citet{Mehdavi12JMLR} and \citet{Jenatton16ICML} consider online convex optimization with long term constraints under Assumptions \ref{as:basic}-\ref{as:bounded-set} and the additional assumption that all $f^{t}(\cdot)$ functions are bounded without imposing Assumption \ref{as:interior-point}.

In this subsection, we show that if we are allowed to introduce the bounded $f^{t}(\cdot)$ assumption (formally defined in Assumption \ref{as:bounded-utility}) used in \citet{Mehdavi12JMLR} and \citet{Jenatton16ICML} , then Algorithm \ref{alg:new-alg} can still achieve a superior performance without imposing Assumptions \ref{as:bounded-set}-\ref{as:interior-point}.

\begin{Assumption}\label{as:bounded-utility}
There exists a constant $F$ such that  $\vert f^{t}(\mathbf{x}) - f^{t}(\mathbf{y})\vert \leq F$ for all $\mathbf{x},\mathbf{y} \in \mathcal{X}_{0}$ and all $t\in\{1,2,\ldots\}$.
\end{Assumption}

Recall that the $O(\sqrt{T})$ regret summarized in Theorem \ref{thm:regret-bound} holds regardless of Assumptions \ref{as:bounded-g}-\ref{as:interior-point}. Now, it remains to show the constraint violation of Algorithm \ref{alg:new-alg} under Assumption \ref{as:basic} and Assumption \ref{as:bounded-g}, both of which are used in \citet{Mehdavi12JMLR} and \citet{Jenatton16ICML}.

\begin{Thm}\label{thm:constraint-bound-no-interior} 
Consider online convex optimization with long term constraints under Assumptions \ref{as:basic}-\ref{as:bounded-g} and \ref{as:bounded-utility}.  If we choose $\gamma = T^{1/4}$ and $\alpha = \frac{1}{2}[\beta^2 +1] \sqrt{T}$ in Algorithm \ref{alg:new-alg}, we have $$\sum_{t=1}^{T} g_{k}(\mathbf{x}(t)) \leq O(T^{1/4}), \forall k\in\{1,2,\ldots,m\}.$$
\end{Thm}
\begin{proof}
Fix $T\geq 1$. Let $\gamma>0, \eta>0$ be arbitrary. If $\alpha \geq \frac{1}{2}[\gamma^{2}\beta^{2} + \eta]$,  by Lemma \ref{lm:dpp-bound}, we have 
\begin{align*}
&\Delta(t) + f^{t}(\mathbf{x}(t)) \\
\leq &f^{t}(\mathbf{x}^{\ast}) + \alpha [\Vert \mathbf{x} ^{\ast}- \mathbf{x}(t)\Vert^{2} - \Vert \mathbf{x}^{\ast} - \mathbf{x}(t+1)\Vert^{2}]  + \frac{1}{2} [ \Vert \tilde{\mathbf{g}}(\mathbf{x}(t+1))\Vert^{2} - \Vert \tilde{\mathbf{g}}(\mathbf{x}(t))\Vert^{2}] + \frac{1}{2\eta} D^{2}.
\end{align*}  
Summing over $t\in\{1,2,\ldots,T-1\}$ and rearranging terms yields
\begin{align*}
\sum_{t=1}^{T-1} \Delta(t) \leq& \sum_{t=1}^{T-1} \big[ f^{t}(\mathbf{x}^{\ast}) - f^{t}(\mathbf{x}^{t}) \big] + \alpha \sum_{t=1}^{T-1}[\Vert \mathbf{x} ^{\ast}- \mathbf{x}(t)\Vert^{2} - \Vert \mathbf{x}^{\ast} - \mathbf{x}(t+1)\Vert^{2}]  \\&+ \sum_{t=1}^{T-1}\frac{1}{2} [ \Vert \tilde{\mathbf{g}}(\mathbf{x}(t+1))\Vert^{2} - \Vert \tilde{\mathbf{g}}(\mathbf{x}(t))\Vert^{2}] + \frac{1}{2\eta} [T-1]D^{2}  \\
=& \sum_{t=1}^{T-1} \big[f^{t}(\mathbf{x}^{\ast}) - f^{t}(\mathbf{x}^{t}) \big] + \alpha \Vert \mathbf{x} ^{\ast}- \mathbf{x}(1)\Vert^{2} -  \alpha \Vert \mathbf{x}^{\ast} - \mathbf{x}(T)\Vert^{2}  \\&+ \frac{1}{2} [ \Vert \tilde{\mathbf{g}}(\mathbf{x}(T))\Vert^{2} - \Vert \tilde{\mathbf{g}}(\mathbf{x}(1))\Vert^{2}] + \frac{1}{2\eta} [T-1] D^{2} 
\end{align*}
Substituting $\Delta(t) = \frac{1}{2} \Vert \mathbf{Q}(t+1)\Vert^{2} - \frac{1}{2} \Vert \mathbf{Q}(t)\Vert^{2}$ into it, simplifying the telescoping sum, and rearranging terms yields 
\begin{align*}
\frac{1}{2} \Vert \mathbf{Q}(T)\Vert^{2} \leq & \sum_{t=1}^{T-1} \big[f^{t}(\mathbf{x}^{\ast}) - f^{t}(\mathbf{x}^{t}) \big] + \alpha \Vert \mathbf{x} ^{\ast}- \mathbf{x}(1)\Vert^{2} -  \alpha \Vert \mathbf{x}^{\ast} - \mathbf{x}(T)\Vert^{2}  \\&+ \frac{1}{2}\Vert \tilde{\mathbf{g}}(\mathbf{x}(T))\Vert^{2} - \frac{1}{2} \Vert \tilde{\mathbf{g}}(\mathbf{x}(1))\Vert^{2} + \frac{1}{2} \Vert \mathbf{Q}(1)\Vert^{2} + \frac{1}{2\eta} [T-1] D^{2} \\
\overset{(a)}{=} &  \sum_{t=1}^{T-1} \big[ f^{t}(\mathbf{x}^{\ast}) - f^{t}(\mathbf{x}^{t}) \big] + \alpha \Vert \mathbf{x} ^{\ast}- \mathbf{x}(1)\Vert^{2} -  \alpha \Vert \mathbf{x}^{\ast} - \mathbf{x}(T)\Vert^{2}  \\&+ \frac{1}{2} \gamma^{2} \Vert \mathbf{g}(\mathbf{x}(T))\Vert^{2} + \frac{1}{2\eta} [T-1] D^{2} \\
\overset{(b)}{\leq}& [T-1]F   + \alpha \Vert \mathbf{x} ^{\ast}- \mathbf{x}(1)\Vert^{2} + \frac{1}{2} \gamma^{2} G^{2} + \frac{1}{2\eta} [T-1] D^{2}
\end{align*}
where (a) follows because $\Vert \tilde{\mathbf{g}}(\mathbf{x}(1))\Vert^{2} = \Vert \mathbf{Q}(1)\Vert^{2}$ by Lemma \ref{lm:virtual-queue}; and (b) follows from Assumption \ref{as:bounded-g} and Assumption \ref{as:bounded-utility}.

Multiplying both sides by a factor of $2$ and taking square roots on both sides yields
\begin{align}
\Vert \mathbf{Q}(T)\Vert \leq \sqrt{2[T-1]F} + \sqrt{2\alpha} \Vert \mathbf{x}^{\ast} - \mathbf{x}(1)\Vert + \gamma G + D\sqrt{\frac{T-1}{\eta}} \label{eq:pf-thm-constraint-bound-no-interior-eq1}
\end{align}

Fix $k\in \{1,2,\ldots, m\}$. By Lemma \ref{lm:queue-constraint-inequality}, we have 
\begin{align*}
\sum_{t=1}^{T} g_{k}(\mathbf{x}(t)) \leq&  \frac{1}{\gamma}Q_{k}(T) \\
\leq&\frac{1}{\gamma} \Vert Q_{k}(T)\Vert\\
\overset{(a)}{\leq}& \frac{\sqrt{2[T-1]F} }{\gamma}+ \frac{\sqrt{2\alpha}}{\gamma} \Vert \mathbf{x}^{\ast} - \mathbf{x}(1)\Vert +  G + \frac{D}{\gamma}\sqrt{\frac{T-1}{\eta}}
\end{align*}
where (a) follows from \eqref{eq:pf-thm-constraint-bound-no-interior-eq1}.

Note that if we let $\gamma = T^{1/4}$ and $\eta = \sqrt{T}$, then $\alpha = \frac{1}{2}[\beta^2+1]\sqrt{T} \geq \frac{1}{2}[\gamma^2\beta^{2} + \eta]$. Substituting these values into the above equation yields
\begin{align*}
\sum_{t=1}^{T} g_{k}(\mathbf{x}(t)) \leq& \frac{\sqrt{2[T-1]F} }{T^{1/4}}+ \frac{\sqrt{[\beta^{2}+1]\sqrt{T}}}{T^{1/4}} \Vert \mathbf{x}^{\ast} - \mathbf{x}(1)\Vert +  G + \frac{D}{T^{1/4}}\sqrt{\frac{T-1}{\sqrt{T}}}\\
=& O(T^{1/4})
\end{align*}
\end{proof}

\begin{Rem}
Theorem \ref{thm:regret-bound} and Theorem \ref{thm:constraint-bound-no-interior} together imply that if we choose $\gamma = T^{1/4}$ and $\alpha = \frac{1}{2}[\beta^2 +1] \sqrt{T}$ in Algorithm \ref{alg:new-alg}, then we can achieve $O(\sqrt{T})$ regret and $O(T^{1/4})$ constraint violations for online convex optimization with long term constraints under Assumptions \ref{as:basic}-\ref{as:bounded-g} and \ref{as:bounded-utility}.  This is still uniformly better than the best known $O(T^{\max\{\theta,1-\theta\}})$ regret and $O(T^{1-\theta/2})$ constraint violations for all $\theta\in (0,1)$ established in \citet{Jenatton16ICML} under Assumptions \ref{as:basic}-\ref{as:bounded-set} and \ref{as:bounded-utility}.  We further note the parameters used to achieve $O(T^{1/4})$ constraint violations under Assumption \ref{as:bounded-utility} are identical to those used in Section \ref{sec:analysis} to achieve $O(1)$ constraint violations under Assumption \ref{as:interior-point}. Thus, Algorithm \ref{alg:new-alg} is very adaptive and its practical implementation can be blind to Assumption \ref{as:interior-point} or Assumption \ref{as:bounded-utility}.
\end{Rem}

\section{Experiment}
This section considers numerical experiments to verify the performance of our algorithm.  Consider online convex optimization with loss functions $f^t (\mathbf{x}) = \mathbf{c}(t)\tran \mathbf{x}$, where $\mathbf{c}(t)$ is time-varying and unknown at round $t$; and constraint functions $\mathbf{A}\mathbf{x} \leq \mathbf{b}$. The constraint functions are only required to be satisfied in long term: 
\begin{align*}
\limsup_{T\rightarrow \infty}\frac{1}{T}\sum_{t=1}^{T}\mathbf{A}\mathbf{x}(t) &\leq \mathbf{b}.
\end{align*}
The above problem formulation arises often in fields such as resource allocation, product planning, finance portfolio selection, network scheduling, load distribution, and so on \citep{book_ResourceAllocation}. For example, consider a power grid network where  the electricity generation at each power plant is scheduled in real-time, e.g., hour-by-hour. In this problem, each component $\mathbf{x}_i$ corresponds to the amount of electricity produced by the $i$-th power plant. The time-varying loss function $f^t(\cdot)$, which represents the economic loss/reward depending on the real-time power demand, is in general unknown to the decision maker at the beginning of round $t$. Inequality constraint $\mathbf{A}\mathbf{x}\leq \mathbf{b}$ corresponds to constraints such as fuel consumption, man-power consumption, carbon emission and electricity scheduling. 

In the numerical experiment, we assume $\mathbf{x}\in \mathbb{R}^2$, $\mathbf{A}\in \mathbb{R}^{3\times 2}$; each component of $\mathbf{x}$ satisfies the box constraint $\mathbf{x} \in \mathcal{X}_0$ where $\mathcal{X}_0 = \{\mathbf{x}: -1 \leq x_1 \leq 1, -1 \leq x_2 \leq 1\}$; and $T=5000$. Each component of $\mathbf{A}$ is generated from the uniform distribution in the interval $[0,1]$ and each component of $\mathbf{b}$ is generated from the uniform distribution in the interval $[0,2]$. $\mathbf{A}$ and $b$ are kept fixed for all rounds once they are generated.   To simulate arbitrarily varying objective functions, at each round $t$, $\mathbf{c}(t) = \mathbf{c}^{(1)}(t) + \mathbf{c}^{(2)}(t) + \mathbf{c}^{(3)}(t)$ is randomly generated such that each component $\mathbf{c}^{(1)}(t)$ is from the uniform distribution in the interval $[-t^{1/10}, +t^{1/10}]$; each component of $\mathbf{c}^{(2)}(t)$ is from the uniform distribution in the interval $[-1, 0]$ when $t\in [1,1500] \cup [2000,3500] \cup [4000,5000]$ and is from the uniform distribution in the interval $[0, 1]$ otherwise; and each element of $\mathbf{c}^{(3)}(t)$ is equal to $(-1)^{\mu(t)}$ where $\mu(t)$ is a random permutation of vector $[1:5000]$ 

We run our proposed Algorithm \ref{alg:new-alg}, our proposed Algorithm \ref{alg:new-alg} with the doubling trick (without knowing $T=5000$),  Algorithm 1 in \citet{Mehdavi12JMLR} and the Algorithm in \citet{Jenatton16ICML} with $\beta =1/2$ and $\beta = 2/3$ over $1000$ independent experiments generated from the above distribution setting. Figure \ref{fig:regret} and Figure \ref{fig:constraint} plot the cumulative regret and the cumulative constraint violations (averaged over $1000$ independent experiments), respectively.  Figure \ref{fig:regret} shows that the first $3$ algorithms have similar regret since they are all proven to have $O(\sqrt{T})$ regret and the Algorithm in \citet{Jenatton16ICML} with $\beta = 2/3$ has the largest regret since it has $O(T^{2/3})$ regret. Figure \ref{fig:constraint} shows that our algorithm has the smallest constraint violation since the constraint violation of our algorithm is bounded by a constant and does not grow with $T$ while the other algorithms have $O(T^{3/4})$ or $O(T^{2/3})$ constraint violations.

\begin{figure}[htbp]
\centering
   \includegraphics[width=0.8\textwidth,height=0.8\textheight,keepaspectratio=true]{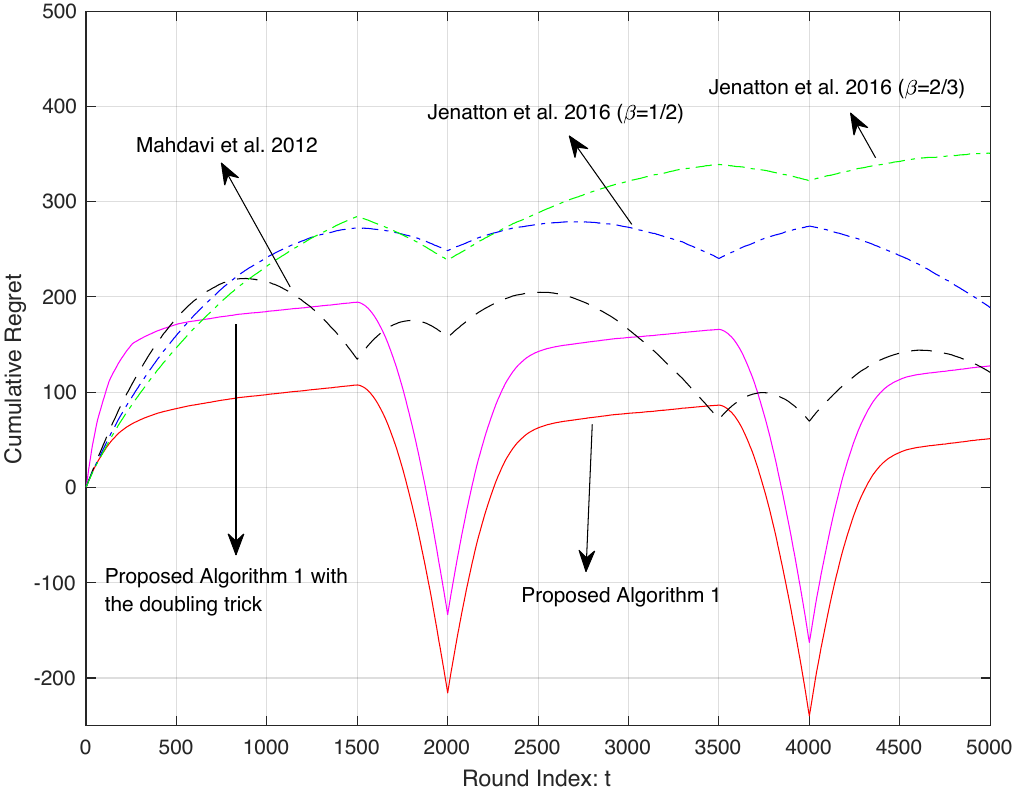} 
   \caption{The cumulative regret.}
   \label{fig:regret}
\end{figure}

\begin{figure}[htbp]
\centering
   \includegraphics[width=0.8\textwidth,height=0.8\textheight,keepaspectratio=true]{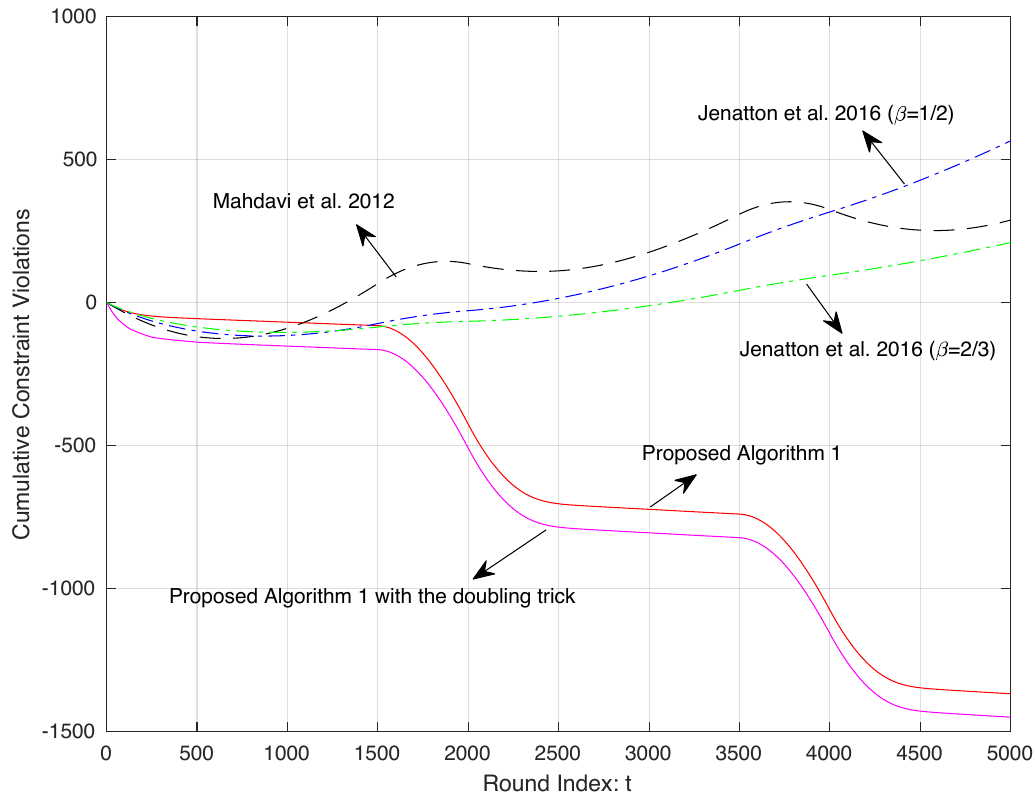} 
   \caption{The cumulative constraint violations of the long term constraints.}
   \label{fig:constraint}
\end{figure}

\section{Conclusion}
This paper considers online convex optimization with long term constraints, where functional constraints are only required to be satisfied in the long term. Prior algorithms in \citet{Mehdavi12JMLR}  can achieve $O(\sqrt{T})$ regret and $O(T^{3/4})$ constraint violations for general problems and achieve $O(T^{2/3})$ bounds for both regret and constraint violations when the constraint set can be described by a finite number of linear constraints.  A recent extension in \citet{Jenatton16ICML} can achieve $O(T^{\max\{\theta,1-\theta\}})$ regret and $O(T^{1-\theta/2})$ constraint violations where $\theta\in (0,1)$. This paper proposes a new algorithm that can achieve an $O(\sqrt{T})$ bound for regret and an $O(1)$ bound for constraint violations; and hence yields improved performance in comparison to the prior works \citep{Mehdavi12JMLR,Jenatton16ICML}.

\section{Acknowledgment}
This work is supported in part by grant NSF CCF-1718477.

\bibliography{mybibfile}

\end{document}